\documentclass[preprint]{imsart}

\RequirePackage[OT1]{fontenc}
\RequirePackage{amsthm,amsmath}
\RequirePackage[numbers]{natbib}
\RequirePackage[colorlinks,citecolor=blue,urlcolor=blue]{hyperref}
\usepackage{amssymb}
\usepackage[english]{babel}

\newcommand{\esProd}[2]{\left\langle #1, #2 \right\rangle} 
\newcommand{\eNorm}[1]{\left\lvert #1 \right\rvert} 
\newcommand{\gNorm}[1]{\left\lvert\left\lvert #1 \right\rvert\right\rvert} 
\newcommand{\wDist}[1]{\mathcal W_{#1}} 

\newcommand{\addDist}{\rho} 
\newcommand{\multDist}{\rho} 

\newcommand{\reals}{\mathbb{R}}

\newcommand{\naturals}{\mathbb{N}}
\newcommand{\contFunctions}{C}
\newcommand{\lypMeasures}{\mathcal{P}_V}

\newcommand{\lypFunc}{V}
\newcommand{\lypConstant}{C}
\newcommand{\lypRate}{\lambda}

\newcommand{\concFunc}{\eta}
\newcommand{\multFunc}{Q}

\newcommand{\generator}{\mathcal{L}}

\renewcommand{\phi}{\varphi}

\arxiv{arXiv:1606.06012}

\startlocaldefs
\numberwithin{equation}{section}
\theoremstyle{plain}
\newtheorem{thm}{Theorem}[section]
\newtheorem{ass}{Assumption}[section]
\newtheorem{rem}{Remark}[section]
\newtheorem{cor}{Corollary}[section]
\newtheorem{lem}{Lemma}[section]
\newtheorem{exa}{Example}[section]

\endlocaldefs

\begin{document}  

\begin{frontmatter}
\thankstext{T1}{Financial support from DAAD and French government through the PROCOPE program, and from the German Science foundation through the {\em Hausdorff Center for Mathematics} is gratefully acknowledged.}
 
\title{Quantitative Harris type theorems for diffusions and McKean-Vlasov processes\thanksref{T1}}
\runtitle{Quantitative Harris type theorems for diffusions}

\begin{aug}
\author{\fnms{Andreas} \snm{Eberle}\ead[label=e1]{eberle@uni-bonn.de}\ead[label=u1,url]{http://wt.iam.uni-bonn.de}},
\author{\fnms{Arnaud} \snm{Guillin}\ead[label=e2]{guillin@math.univ-bpclermont.fr}\ead[label=u2,url]{http://math.univ-bpclermont.fr/~guillin/}}
\and 
\author{\fnms{Raphael} \snm{Zimmer}
\ead[label=e3]{raphael.zimmer@uni-bonn.de} 
\ead[label=u3,url]{http://wt.iam.uni-bonn.de}}

\runauthor{A. Eberle, A. Guillin, R. Zimmer}

\affiliation{University of Bonn\thanksmark{m1} and Universit\'e Blaise Pascal\,\thanksmark{m2}}

\address{Universit\"at Bonn\\
Institut f\"ur Angewandte Mathematik\\
 Endenicher Allee 60\\
  53115 Bonn, Germany\\
\printead{e1}\\
\phantom{E-mail:\ }\printead*{e3}\\
\printead{u1}}

\address{Laboratoire de Math\'ematiques\\
CNRS - UMR 6620\\
Universit\'e Blaise Pascal\\
Avenue des landais,\\
63177 Aubiere cedex, France\\
\printead{e2}\\
\printead{u2}}
\end{aug}

 \begin{abstract}:
We consider $\reals^d$-valued  diffusion processes of type
	\begin{align*} 
	dX_t\ =\ b(X_t)dt\, +\, dB_t.
	\end{align*}   
Assuming a geometric drift condition, we
	establish contractions of the transitions kernels in
	Kantorovich ($L^1$ Wasserstein) distances with explicit
	constants. Our results are in the spirit of
	Hairer and Mattingly's extension of Harris' Theorem. In particular,  
they do not rely on a
	small set condition. Instead we combine Lyapunov functions with 
	reflection coupling and concave distance functions.
We retrieve constants
	that are explicit in parameters which can be
	computed with little effort from one-sided Lipschitz conditions for the drift coefficient and the growth of a chosen Lyapunov function. 
Consequences include exponential convergence in weighted total variation norms, gradient bounds, bounds for ergodic averages, and Kantorovich contractions for nonlinear	McKean-Vlasov diffusions in the case of sufficiently weak but not necessarily bounded nonlinearities. 
We also establish quantitative bounds for sub-geometric ergodicity assuming  a sub-geometric drift condition.
\end{abstract}

\begin{keyword}[class=MSC]
\kwd[Primary ]{60J60}
\kwd{60H10}
\end{keyword}

\begin{keyword}
\kwd{Couplings} 
\kwd{Wasserstein distances}
\kwd{Lyapunov functions}
\kwd{Harris theorem}
\kwd{quantitative bounds}
\kwd{convergence to stationarity}
\kwd{nonlinear diffusions}
\end{keyword}

\end{frontmatter}

\section{Introduction} 
We consider $\reals^d$ valued diffusion processes of
type  
\begin{align}   
	dX_t\ =\ b(X_t)\, dt\, + \, dB_t,\label{eqStdDiffusion}
\end{align}  
where $b:\reals^d\rightarrow\reals^d$ is locally Lipschitz, 
and $(B_t)$ is a
$d$-dimensional Brownian motion. We assume non-explosiveness, and we denote the transition function of the corresponding Markov process by $(p_t)$. 

The classical Harris' Theorem \cite{MR0084889,MR2509253} gives simple conditions for
geometric ergodicity of Markov processes. In the case of diffusion processes on $\reals^d$ it goes back to
Khasminskii \cite{MR0133871,MR2894052},
in the general case it has been developed systematically by Meyn and Tweedie \cite{MR1174380, MR1234295,MR2509253}.
For solutions of \eqref{eqStdDiffusion}, it is often not difficult to verify the 
assumptions in Harris' Theorem, a minorization condition for the transition probabilities on a bounded set, and a global Lyapunov type drift condition.
However, it is not at all easy to quantify the constants in Harris' Theorem, and,
even worse, the resulting bounds are far from sharp, and they usually have
a very bad dimensional dependence. Therefore,
although Harris' Theorem has become a standard tool in many 
application areas, it is mostly used in a 
purely qualitative way, a noteworthy exception being Roberts and Rosenthal
\cite{MR1423462}.

Besides the Harris' approach, there is a standard approach for studying
mixing properties of Markov processes based on spectral gaps, logarithmic
Sobolev inequalities, and more general functional inequalities, see for example the monograph
\cite{MR3155209}. This approach has the advantage of providing
sharp bounds in simple model cases but it sometimes yields slightly 
weaker, and less probabilistically intuitive results. Recent attempts \cite{MR2381160,MR2386063} to 
connect these functional inequalities to Lyapunov conditions have proven successful but 
they are clearly restricted to the reversible setting (or the explicit knowledge of the 
invariant measure). 
The concept of the intrinsic curvature of a diffusion
process in the sense of Bakry-Emery leads to sharp bounds and many
powerful results in the case where there is a strictly positive lower curvature
bound $\kappa$ \cite{MR2142879}. In our context, this means that $\partial b (x)\le -\kappa \,I_d$ for all $x$ in the sense of quadratic forms.

Several of the 
bounds in the positive curvature case can be derived in an elegant 
probabilistic way by considering synchronous couplings and contraction properties in $L^2$ Wasserstein distances.
In general, Wasserstein distances have proved crucial in the study of linear and nonlinear diffusions both via coupling 
techniques \cite{MR972776,MR1970276,MR2357669}, or via analytic techniques based on 
profound results on optimal transportation, see \cite{MR2053570,MR2209130,MR2459454,MR2964689,MR3035983} 
and references therein. In the case where the curvature is only strictly positive outside of a compact set, reflection coupling has been applied successfully to obtain total variation
bounds for the distance to equilibrium \cite{MR841588} as well
as explicit contraction rates of the transition semigroup in Kantorovich distances \cite{MR2843007,Eberle2015}.

An important question is how to apply a Harris' type approach
in order to obtain explicit bounds that are close to sharp in certain 
contexts. A breakthrough towards the applicability 
to high- and infinite dimensional models has been made by Hairer and Mattingly in
\cite{MR2478676}, and in the subsequent publications \cite{MR2773030,MR2857021}. The key idea was to replace the 
underlying couplings with finite coupling time by asymptotic couplings where
the coupled processes only approach each other as $t\to\infty$
\cite{MR1939651,MR1937652}, and the minorization condition by a contraction in an appropriately chosen Kantorovich distance. 
In recent years, the resulting weak Harris' theorem
has been applied successfully to prove (sub)geometric ergodicity in infinite dimensional models (see e.g.\ \cite{MR3178490}), and to quantify the dimension dependence in high dimensional
problems \cite{MR3262508}.
Nevertheless, in contrast to the approach based on functional inequalities,
the constants in applications of the weak Harris theorem are usually far
from optimal. This is in particular due to the fact that the corresponding
Kantorovich distance is still chosen in a somehow ad hoc way.

It turns out that a key for making the bounds more quantitative 
is to adapt the underlying metric on $\reals^d$ and the corresponding Kantorovich
metric on the space of probability measures in a very specific way to the problem under consideration. For diffusion processes solving \eqref{eqStdDiffusion},
this approach has been discussed in \cite{Eberle2015} assuming strict contractivity for the corresponding deterministic dynamics
outside a ball. Our goal here is to replace this ``Contractivity at infinity''
condition by a Lyapunov condition, thus providing a more specific
quantitative version of the weak Harris' theorem. Indeed, we will
define explicit metrics on $\reals^d$ depending both on the drift coefficient $b$ and
the Lyapunov function $V$ such that the transition semigroup is
contractive with an explicit contraction rate $c$ w.r.t.\ the corresponding
Kantorovich distances. Such a contraction property has
remarkable consequences including not only a non-asymptotic
quantification of the distance to equilibrium, but also non-asymptotic
bounds for ergodic averages, gradient bounds for the transition
semigroup, stability under perturbations etc.\medskip

In Section \ref{sec:main}, we present 
our main results. 
Section \ref{secDiscussion}
contains a discussion of the results including more detailed comparisons to the existing literature.
The couplings considered here are introduced in Section \ref{secCoupling}, and the proofs 
of our results are given in Section \ref{secProofs}.


\section{Main results}\label{sec:main}

Let $\esProd{\cdot}{\cdot}$ and $\eNorm{\cdot}$, respectively, denote
the euclidean inner product and the corresponding norm on $\reals^d$. 
We assume a generalization of a global one-sided Lipschitz condition for $b$ combined with a Lyapunov condition. These assumptions can be weakened - we refer to \cite{Zimmer} for an extension of the results to
a more general setup.

\begin{ass}[Generalized one-sided Lipschitz condition]\label{assGlobalCurvBound}
There is a continuous function $\kappa :(0,\infty )\rightarrow
[0,\infty)$ such that $\int_0^1 r\, \kappa (r)\, dr \, <\, \infty$, and
	\begin{equation}\label{eq:kappa}
	\esProd{x-y}{b(x)-b(y)}\ \leq\ \kappa (|x-y|)\cdot
		\eNorm{x-y}^2 \qquad \text{for any } x,y\in\reals^d,\, x\not=y.
	\end{equation}
\end{ass}

Notice that the one-sided condition (\ref{eq:kappa}) does not imply regularity of $b$. For constant $\kappa$,  it is a one-sided Lipschitz condition. In particular, if $b=-\nabla U$ for some function $U\in C^2(\reals^d)$ then
the assumption with constant $\kappa$ is equivalent to a global lower bound on the Hessian of $U$. If $U$ is strictly convex outside a ball in $\reals^d$ then we can choose $\kappa (r)=0$ in (\ref{eq:kappa}) for sufficiently large
$r$. Let
$$\mathcal{L}\, =\, \frac{1}{2}\Delta\,+\, \esProd{b(x)}{\nabla}$$
denote the generator of the diffusion process.

\begin{ass}[Geometric drift condition]\label{assGeometricDrift}
There is a $\contFunctions^2$ function $\lypFunc:\reals^d\rightarrow \reals_+$
as well as constants $C,\lambda\in (0,\infty )$ such that $\lypFunc(x)\rightarrow \infty$ as $\eNorm{x}\rightarrow\infty $, and
\begin{equation}\label{eq:Lyapunov}
\generator \lypFunc(x)\ \leq\ \lypConstant-\lypRate  \lypFunc(x)\qquad
\mbox{for any }x\in\reals^d.
\end{equation}  
\end{ass}

It is well-known that Assumption \ref{assGeometricDrift} implies
the non-explosiveness of solutions for \eqref{eqStdDiffusion}, see e.g. \cite{MR2894052,MR1234295}. 
The function $V$ in \eqref{eq:Lyapunov} is called a {\em Lyapunov function}.

\begin{rem}[Choice of Lyapunov functions] \label{additive_example_lyapunov}
	Assume there are $\mathcal R>0$, $\gamma>0$ and $q\geq 1$ such that
	\begin{align*}\esProd{b(x)}{x}\leq - \gamma
	\eNorm{x}^q\qquad\mbox{for any } \eNorm{x}\geq \mathcal R .\end{align*}
	Then Lyapunov functions of the following form can be chosen depending on the values of 
	$q$ and $\gamma$:
	\begin{itemize}
	\item Let $\alpha >0$. If $V$ is a $\contFunctions^2$ function with $V(x)=\exp(\alpha \eNorm{x}^q)$ outside of a compact set,  then
	\eqref{eq:Lyapunov} holds for arbitrary $\lambda >0$ with a finite constant $C(\alpha ,\lambda )$
	provided $q>1$ and $\alpha\in (0,2\gamma /q)$, or $q=1$ and $\gamma>\frac{\alpha}{2}+\frac{\lambda}{\alpha}$.
	\item Let $\alpha>0$
	and $p\in [1,q)$. If $V$ is $\contFunctions^2$ with $V(x)=\exp(\alpha \eNorm{x}^p)$ outside of a compact set, then \eqref{eq:Lyapunov} holds for arbitrary $\lambda >0$ with a finite constant $C(\alpha ,\lambda )$.
	\item Let $q\geq 2$ and $p>0$. If $V$ is $\contFunctions^2$ with $V(x)=|x|^p$ outside of a compact set,  
	then \eqref{eq:Lyapunov} holds with a finite constant $C(\lambda,p)$ if
	$q>2$ and $\lambda >0$, or if $q=2$ and $\lambda\in (0,p \gamma )$.
	\end{itemize}
\end{rem}

Besides the two key assumptions made above, we
will need a growth assumption on the Lyapunov function, 
cf.\ Assumption \ref{assAddGrowthCurvBound} below for our first main result, or Assumption \ref{assGeometricDriftAddOn} below for our second main result. 
\medskip

The Kantorovich distance of two probability measures $\mu$ and $\nu$ 
on a metric space $(S,\rho )$ is defined by
$$\wDist{\rho}(\nu,\mu)\ =\ \inf_{\gamma\in C(\nu ,\mu)} \int \rho (x,y) \; \gamma (dx\, dy)$$
where the infinum is taken over all couplings with marginals $\nu$ and $\mu$ respectively. $\mathcal W_\rho (\nu ,\mu )$ can be interpreted  
as the $L^1$ transportation cost between the probability measures $\nu$ and 
$\mu$ w.r.t.\ the underlying cost function $\rho (x,y)$. As such, it is also
well-defined if $\rho $ is only a semimetric, i.e., a function on $S\times S$ that is symmetric and non-negative with
$\rho (x,y)>0$ for $x\neq y$
but that does not necessarily satisfy the triangle inequality.
In Subsections \ref{secGeometricSimple} and \ref{secMultiplicative} we derive contractions of the transition semigroup  with respect to
$\mathcal W_\rho$ based on two different types
of underlying cost functions $\rho$. The first one, called the ``additive distance'', is a metric, whereas the second one, called the
``multiplicative distance'', in general is only a semimetric.
We then consider a variation of our approach that
applies to McKean-Vlasov diffusions, cf.\ Subsection \ref{secMckeanvlasov}. Subsection \ref{secSubGeometricSimple} discusses
replacing the geometric by a subgeometric Lyapunov condition.

\subsection{Geometric ergodicity with explicit constants: First main result}\label{secGeometricSimple}


We first consider an underlying distance of the
form
\begin{align}
\rho_1(x,y):=\left[\, f(\eNorm{x-y})+\epsilon \, V(x)+\epsilon\, 
		V(y)\, \right]\cdot I_{x\not = y}\label{addDistance}
\end{align}
where $f$ is a suitable bounded, non-decreasing and concave
continuous function satisfying $f(0)=0$, and $\epsilon>0$ is a positive constant. The choice of a
distance is partially motivated by \cite{MR2857021} where an
underlying metric of the form $(x,y)\mapsto[2+\epsilon
V(x)+\epsilon V(y)]\, I_{x\not= y} $ is considered in order to retrieve a Kantorovich contraction
based on a small set condition.  We define functions 
$\phi,\Phi:\reals_+\rightarrow\reals_+$ by
\begin{align}\label{eq:phi}
	\phi(r)=\exp\left(\,-\frac{1}{2} \int_0^r t\, \kappa(t)\, dt\,\right) & &\text{and} & &\Phi(r)=\int_0^r \phi(t) \,dt
\end{align}
with $\kappa$ as in Assumption \ref{assGlobalCurvBound}.
For constant $\kappa$, we have $\phi(r)=\exp(-{\kappa}r^2/4)$. 
We will choose the function $f$ to be constant outside a finite interval $[0,R_2]$ where $R_2$ is defined in (\ref{R2additivesimple}) below. Inside the interval, the function $f$ 
will satisfy 
$$\frac{1}{2}\Phi(r)\leq f(r) \leq \Phi(r).$$
We consider a set $S_1$ which is recurrent for any Markovian coupling $(X_t,Y_t)$ of solutions of \eqref{eqStdDiffusion}:
\begin{equation}\label{eq:S1}
S_1\ :=\ \left\{(x,y)\in \reals^d\times\reals^d:
\lypFunc(x)+\lypFunc(y)\leq {4\lypConstant}/{\lambda}\right\} .
\end{equation}

The set $S_1$ is chosen such that for $(x,y)\in\reals^{2d}\setminus S_1$, 
\begin{equation}\label{eq:Lyapunovxy}
\generator\,	\lypFunc(x)+\generator \,\lypFunc(y)\ \leq\  -\frac{\lambda}{2} \, \left(\, \lypFunc(x)+\lypFunc(y)\, \right).
\end{equation}
Here the factor $1/2$ is,
	to some extent, an arbitrary choice. The ``diameter''
	\begin{equation}
R_1\ :=\ \sup\left\{ \eNorm{x-y}:(x,y)\in S_1\right\}\label{R1additivesimple}
\end{equation}
of the set $S_1$ determines our choice of $\epsilon$ in \eqref{addDistance}~:
	\begin{equation}\label{eq:epsilon}
	\epsilon^{-1}\ :=\ 4C\,\int_0^{R_1} \phi(r)^{-1} \, dr\ =\ 4C\,\int_0^{R_1} \exp\left(\,\frac{1}{2} \int_0^r t\, \kappa(t)\, dt\,\right) \, dr.
	\end{equation}
Notice that $R_1$ is always finite
since $V(x)\rightarrow \infty$ as $\eNorm{x}\rightarrow\infty$.
An upper bound is given by 
\begin{eqnarray}\label{R1trivialupperbound}
R_1 &\leq& 2\, \sup\{\, \eNorm{x}: V(x)\le
	4\lypConstant/\lambda\}.
\end{eqnarray}

We now state our third key assumption that links $\kappa$ and $V$:

\begin{ass}[Growth condition]\label{assAddGrowthCurvBound}
There exist a constant $\alpha>0$ and a bounded set $S_2\supseteq S_1$ such that for any
	$(x,y)\in  \reals^{2d}\setminus S_2$, we have
	 \begin{align}   \label{eq:smplAddSetS2}
	 	\lypFunc(x)+\lypFunc(y)&\ \geq\ 
	 	\frac{4C}{\lambda}\, \left(1+\alpha \, \int_0^{R_1}  \phi(r)^{-1} \, dr \; \Phi(\eNorm{x-y})\right).
	\end{align}
\end{ass} 
Assumptions linking curvature and the Lyapunov function already appeared in \cite{MR2653230} to prove a logarithmic Sobolev inequality in the reversible setting in the case where the curvature may explode (polynomially). 
Similarly to $R_1$ we define
\begin{equation}
R_2\ :=\ \sup\left\{ \eNorm{x-y}:(x,y)\in S_2\right\}.\label{R2additivesimple}
\end{equation}
Note that $\Phi$ grows at most linearly. 
If one chooses $\alpha^{-1}=\int_0^{R_1} \, \phi(r)^{-1} \, dr$, then Condition
\eqref{eq:smplAddSetS2} takes the form
$$\lypFunc(x)+\lypFunc(y)\ \geq\
	 	\frac{4C}{\lambda}\,\left(1+ \int_0^{\eNorm{x-y}} \exp\left(\,-\frac{1}{2} \int_0^r t\, \kappa(t)\, dt\,\right) dr\right).$$

\begin{lem}
\label{lemSufficientCriteriaAddCase1}
	If there exists a finite constant $\mathcal R\geq R_1$ such that 
\begin{align*}
		V(x)\geq {4C}{\lambda}^{-1}\left(1+ 2 \, \eNorm{x}\right)\qquad
		\mbox{for any }x\in\reals^d\mbox{ with }\eNorm{x}\geq \mathcal R,
\end{align*}
then Assumption \ref{assAddGrowthCurvBound} is satisfied with $\alpha^{-1}=\int_0^{R_1} \, \phi(r)^{-1} \, dr$ and
$$S_2\ =\ \{(x,y)\in\reals^d\times\reals^d:
\max\{\eNorm{x},\eNorm{y}\}<  \mathcal R\}.$$
\end{lem}

The proofs for Lemma \ref{lemSufficientCriteriaAddCase1} and the subsequent results in Section \ref{secGeometricSimple}
are given in Section \ref{secProofsGeometricSimple} below.

\begin{exa}[Exponential tails]\label{exa:exptails}
Let $b_r(x)=b(x)\cdot x/|x|$ denote the radial component of the drift. Suppose that 
there is a constant $\delta >0$ such that
\begin{equation}
\label{constdrift}
b_r(x)\ \le\ -\delta\qquad\mbox{for any }x\in\mathbb R^d\mbox{ with } |x|\ge 2d/\delta ,
\end{equation}
and let $\alpha :=d/\delta $. Then the Lyapunov function 
$$V(x)\ =\ \exp (\alpha h(|x|)),\qquad h(r)\ =\ \begin{cases}r&\mbox{for }r\ge 2/\alpha,\\ \alpha^{-1}+\alpha r^2/4 &\mbox{for }r\le 2/\alpha,\end{cases}$$
satisfies the geometric drift condition \eqref{eq:Lyapunov} with constants
\begin{equation}
\label{Clambda}
C\ =\ \frac{e^2}{2}\, \frac{\delta^2}{d},\qquad\lambda\ =\ \frac 14\, \frac{\delta^2}{d},\qquad 4C\lambda^{-1 }\ =\ 8e^2.
\end{equation}
For this choice of $V$, \eqref{R1trivialupperbound} implies $R_1\le \bar R_1$ where
\begin{equation}
\label{BR1}\bar R_1\ :=\ 2\alpha^{-1}\log (4C\lambda^{-1})\ =\ 2\log (8e^2)\, d/\delta .
\end{equation}
Furthermore, by Lemma \ref{lemSufficientCriteriaAddCase1}, we can choose the set $S_2$ such that
\begin{equation}
\label{BR2}\bar R_2\ \le\ 2\, (1+\bar R_1)\, \log (1+\bar R_1),
\end{equation}
see Section  \ref{secProofsGeometricSimple} for details. 
\end{exa}

Let  $\lypMeasures$ denote the set of all probability measures $\mu$
on $\reals^d$ such that $\int V\, d\mu <\infty $. We can now state our first main result:

\begin{thm}[Contraction rates for additive metric]\label{thmAddMainSimple}
Suppose that Assumptions  \ref{assGlobalCurvBound}, \ref{assGeometricDrift}  and 
\ref{assAddGrowthCurvBound} hold true. Then there exist a
concave, bounded and non-decreasing continuous function $f:\reals_+\rightarrow \reals_+$
with $f(0)=0$ and constants $c,\epsilon\in (0,\infty )$ s.t.
\begin{align}\label{eq_contrAddSimple}
	\wDist{\rho_1}(\mu p_t,\nu p_t)&\leq e^{-c\, t}\, 
	\wDist{\rho_1}(\mu,\nu)\qquad\mbox{for any }\mu,\nu\in \lypMeasures .
\end{align} 
Here the underlying distance $\rho_1$ is defined by \eqref{addDistance} with $\epsilon$ determined by \eqref{eq:epsilon}, and
	$c=\min\left\{\beta ,\alpha ,\lambda\right\}/2$ where 
	\begin{equation}\label{defbeta}
			\beta^{-1}\ :=\ \int_0^{R_2} \Phi(r)\, \phi(r)^{-1} ds = \int_0^{R_2} \int_0^s \exp\left(\frac{1}{2}\int_r^s u\, \kappa(u) \,
		du\right) \, dr\,  ds.
	\end{equation}
	The function $f$ is constant for $r\geq R_2$, and
	\begin{align*}
		\frac{1}{2} \ \leq\ f'(r) \, \exp\left(\,\frac{1}{2} \int_0^r t\, \kappa(t)\, dt\,\right)\ \leq\ 1\qquad\mbox{ for any }r\in(0,R_2) .
	\end{align*}
		The precise definition of the function $f$ is given in the proof in Section
	\ref{secProofsGeometricSimple}.
\end{thm}
 
\begin{exa}[Bounds under global one-sided Lipschitz condition]\label{assConstCurvBound}\label{simplifiedconstants}
	Suppose that there is a {\em constant} $\kappa\geq 0$ such that for any $x,y\in\reals^d$, we have
	\begin{equation}
	\label{eq:onesidelip}
	\esProd{x-y}{b(x)-b(y)}\ \leq\ \kappa \eNorm{x-y}^2.
	\end{equation}
Then we can state our result in a simplified form.
Suppose that Assumption \ref{assGeometricDrift} holds, and there is a bounded set $S_2\supseteq
	S_1$ such that for any $(x,y)\not\in S_2$,
	 \begin{align*}   
	 	&\lypFunc(x)+\lypFunc(y)\ \geq\
	 	\frac{4C}{\lambda}\left(1+\int_0^{\eNorm{x-y}}
	 	\exp\left(-{\kappa}r^2/4\right) dr\right) 
	\end{align*}
Then \eqref{eq_contrAddSimple} holds with
 $\ c = \min\left\{{2}{R_2^{-2}},{R_1^{-1}},\lambda\right\}/{2}$ for $ \kappa=0$,  and
	\begin{equation}
c\, =\,\frac{1}{2} \min\left\{\sqrt{\frac{\kappa}{\pi}} \left(\int_0^{R_2} \exp\left(\kappa r^2/4\right)\,dr\right)^{-1},\, \left(\int_0^{R_1} \exp\left({\kappa}r^2/4\right)
	 dr\right)^{-1},\, \lambda\right\}
\label{eq:contractionrate}	
	\end{equation}
 for $\kappa>0$. 
	Here $R_1$ and $R_2$ are defined as above, and the underlying distance $\rho_1$ is given by \eqref{addDistance} with
 $\epsilon^{-1}  = {4C\int_0^{R_1} \exp\left({\kappa}s^2/4\right)
	 ds}$	and a
concave, bounded and increasing continuous function $f:\reals_+\rightarrow \reals_+$
satisfying $f(0)=0$ and
	\begin{align*}
		&{1}/{2}\ \leq\ f'(r)\exp\left({\kappa}r^2/4 \right)
		\ \leq\ 1\qquad\mbox{for }0<r<R_2.
	\end{align*}
For example, suppose that the drift condition \eqref{constdrift} holds and $\kappa =0$. Then by \eqref{BR1} and \eqref{BR2}, we obtain contractivity with rate
$$c\ \ge\ \frac 14(1+\bar R_1)^2\log^2(1+\bar R_1)\qquad\mbox{where}\quad \bar R_1
\ =\ 2\log (8e^2)d/\delta .$$
An example is the exponential distribution $\mu (dx)\propto e^{-|x|}dx$ on $\mathbb R^d$. Here the spectral gap is known to be of order $\Theta (d^{-1})$, see the remark after Theorem 1 in \cite{bobkov}, whereas our lower bound for the contraction rate 
is of order $\Omega (d^{-2}\log^{-2}( d))$. On the contrary, for $\kappa >0$, our lower bound for the contraction rate depends exponentially on the dimension. This is unavoidable in the general setup considered here.
\end{exa}

\begin{rem}
For $\kappa =0$ and, more generally, for $\kappa R_2^{2}=O(1)$, the lower bound $c$ for the contraction rate in the example is of the
optimal order
$\Omega (\min \{ R_2^{-2},\lambda \})$ in $R_2$ and $\lambda$. In general, under the assumptions made above, the bound on the
contraction rate given by \eqref{eq:contractionrate} is of optimal
order in $\lambda$, and of optimal order in $R_2$ up to
polynomial factors, see the discussion below Lemma 1 in
\cite{Eberle2015}. 

The high dimensional exponential distribution $\mu (dx)\propto e^{-|x|}dx$ concentrates in an $O(\sqrt d)$ neighborhood of a sphere of radius $O(d)$ in $\mathbb R^d$. Therefore, the spectral gap is of order $d^{-1}=(\sqrt d)^{-2}$, and not of order $d^{-2}$ as one might naively expect. In contrast, our approach can only take into account the concentration of the measure on a ball of radius $O(d)$ (modulo logarithmic corrections), and therefore it is not able to recover a better rate than $O(d^{-2})$
in this example.
\end{rem}

It is well-known (see e.g.\ \cite{hairer2006ergodic}), that the local Lipschitz assumption on $b$ and Assumption \ref{assGeometricDrift} imply that
$(p_t)$ has a unique invariant measure $\pi\in\mathcal P_V$ satisfying 
$\int V\, d\pi\leq C/\lambda$.
A result from \cite[Lemma 2.1]{MR2857021} then shows that the 
Kantorovich contraction in Theorem \ref{thmAddMainSimple} implies bounds for the distance of $\mu
p_t$ and $\pi $ in a weighted total variation norm.

\begin{cor}[Exponential Convergence in Weighted Total Variation Norm]\label{corAddSimple2}  Under the assumptions of Theorem
\ref{thmAddMainSimple}, there exists a unique stationary distribution 
$\pi\in\mathcal P_V$, and 
\begin{equation}
 \int_{\reals^d} V \, d\eNorm{\mu p_t-\pi}\ \leq\ \epsilon^{-1}\, \exp(-c\,t) \, \wDist{\addDist_1}(\mu,\pi)\quad \mbox{for any }\mu\in \lypMeasures .
\end{equation}
In particular, for any $\delta>0$ and $x\in\reals^d$, the mixing time
$$\tau(\delta ,x) \ :=\ \inf\{t\geq 0 \, : \,\int_{\reals^d} V \, d\eNorm{
 p_t(x,\bullet )-\pi } < \delta \, \}$$
is bounded from above by
$$\tau(\delta ,x)\ \leq\ c^{-1} \, \log^+\left[\frac{R_2\,\epsilon^{-1}+ V(x)+{C}/{\lambda} }{\delta}\right]\, .$$ 
\end{cor}

\begin{rem}[Exponential Convergence in $L^p$-Wasserstein distances]\label{rem:expConvInLp} $\text{ }$\newline For $p\in[1,\infty)$,
the standard $L^p$-Wasserstein distance $\mathcal W^p$ can be
controlled by a weighted total variation norm:
$$\wDist{}^p(\mu , \nu )\ \leq\ 2^{1/q} \, \left(\int
\eNorm{x}^p \,\eNorm{\mu  - \nu }(dx)\right)^{1/p},$$
where $1/q+1/p=1$, see e.g.
\cite[Theorem 6.15]{MR2459454}. Thus if there is a constant $K>0$ such that $\eNorm{x}^p\leq K \, 
V(x)$ holds for all $x\in\reals^d$, then Corollary \ref{corAddSimple2}
also implies exponential convergence in $L^p$-Wasserstein distance. 
\end{rem} 

Following ideas from \cite{MR2543873,MR2683634,Eberle2015},  we show that Theorem \ref{thmAddMainSimple} can be used to control
the bias and the variance of ergodic averages. Let
\begin{equation}\label{eq:Lipschitznorm}
\|{g}\|_{\operatorname{Lip}(\rho )}\ =\ \sup\left\{\eNorm{g(x)-g(y)}/{\rho (x,y)}: \,x,y\in\reals^d, \,x\not=y\right\}
\end{equation}
denote the Lipschitz seminorm of a measurable function $g:\reals^d\rightarrow\reals$ w.r.t.\ a metric $\rho $.

\begin{cor}[Ergodic averages]\label{cor:ergodicAvergaesAdditive}
	Suppose that the assumptions of Theorem \ref{thmAddMainSimple} hold true. Then for any $x\in\reals^d$ and $t>0$,
	 \begin{align*}
	 		\eNorm{E_x\left[\frac{1}{t}\int_0^t g(X_s) \,ds-\int g \, d\pi\right]}\leq
	 		\frac{1-e^{-c\,t}}{c\,t}\|g\|_{\operatorname{Lip}(\addDist_1)} \left(R_2+\epsilon \, V(x)+\epsilon \, \frac{C}{\lambda}\right). \end{align*}
	 If, moreover, the function $x\mapsto V(x)^2$ satisfies the geometric drift condition
	 \begin{equation}\label{eq:LyapunovQuadrat}
(\generator \lypFunc^2)(x)\ \leq\ C^\ast-\lambda^\ast  \lypFunc(x)^2\qquad
\mbox{for any }x\in\reals^d
\end{equation} 
 with constants ${C}^\ast ,\lambda^\ast\in (0,\infty )$, then
	 \begin{align*}
	 		\operatorname{Var}_x\left[\frac{1}{t}\int_0^t g(X_s) ds\right]&\ \le\
	\frac{3}{c\, t} \| g\|_{\operatorname{Lip}(\rho_1)}^2  \left(\,R_2^2+2\epsilon^2  \left[{C^\ast }/{\lambda^\ast } +e^{-{\lambda^\ast}\, t} V(x)^2 \right] \right).
	 \end{align*}
\end{cor}

\subsection{Geometric ergodicity with explicit constants: Second main result}\label{secMultiplicative}
~\newline The additive distance $\mathcal W_{\rho_1}$ defined in \eqref{addDistance} is very simple, and contractivity w.r.t.\ 
$\mathcal W_{\rho_1}$ even implies bounds in
weighted total variation norms. However, this
distance has the disadvantage that in general
$\rho_1(x,y)\not\rightarrow 0$ as $x\rightarrow y$. 
Therefore, a contraction w.r.t.\ $\mathcal W_{\rho_1}$ as stated in \eqref{eq_contrAddSimple}
can only be expected to hold if there is a coupling
$(X_t,Y_t)$ such that $P(X_t=Y_t)\rightarrow 1$ as $t\rightarrow\infty$.
In the case of
non-degenerate diffusions as in \eqref{eqStdDiffusion}, it is not difficult to
construct such a coupling, but for generalizations to infinite dimensional or
nonlinear diffusions, such couplings might not be natural, see e.g.\ \cite{zimmer16} and
Section \ref{secMckeanvlasov} below.

Partially motivated by the weak Harris Theorem in \cite{MR2773030}, we will now replace the additive metric by a
multiplicative semimetric. This will allow us to derive quantitative
bounds for asymptotic couplings in the sense of \cite{MR1939651,MR2773030}, i.e.,
couplings for which
$X_t$ and $Y_t$ get arbitrarily close to each other but do not
necessarily meet in finite time. To this end we consider
an underlying distance-like function
\begin{align}
\rho_2(x,y)\ =\ f(\eNorm{x-y})\left(1+\epsilon V(x)+\epsilon
		V(y)\right)\label{multDist}
\end{align}
where $f$ is a suitable, non-decreasing, bounded and concave continuous function satisfying
$f(0)=0$. 
	Note that in general, the function
	$\multDist_2$ is a semimetric but not necessarily a metric, since the triangle inequality may be violated. Nevertheless, the Lipschitz norm w.r.t.\ $\rho_2$ is still well-defined by (\ref{eq:Lipschitznorm}).
	In \cite[Lemma 4.14]{MR2773030}, conditions are given under which
	$\multDist_2$ satisfies a weak triangle inequality, i.e., under which
	there is a constant $C>0$ such that $\multDist(x,z)\leq C (\multDist(x,y)+\multDist(y,z))$ holds for all
$x,y,z\in\reals^d$. This is for example the case if $V$ is strictly positive and grows at most polynomially, or
if $V(x)=\exp(\alpha \eNorm{x})$ for large $\eNorm{x}$.
In any case, $\rho_2$ has the desirable property that $\rho_2(x,y)\rightarrow 0$ as $x\rightarrow y$.\medskip
  
As before, we assume that Assumptions \ref{assGlobalCurvBound} and \ref{assGeometricDrift} hold true. 
The growth condition on the Lyapunov function in Assumption \ref{assAddGrowthCurvBound}
is now replaced by the following condition:

\begin{ass}\label{assGeometricDriftAddOn}
The logarithm of $V$ is Lipschitz continuous, i.e., 
$$\sup\frac{ |\nabla V|}{V}<\infty.$$
\end{ass}
\noindent
Notice that in contrast to Assumption  \ref{assAddGrowthCurvBound}, Assumption \ref{assGeometricDriftAddOn} does not depend on $\kappa$.
The global bound on $\nabla V$ can be replaced by a local bound, see \cite{Zimmer} for an extension.
The second part in Assumption \ref{assGeometricDriftAddOn} is satisfied if, for example, $V$ is strictly positive,
and, outside of a compact set, $V(x)=|x|^\alpha$ or
	$V(x)=\exp(\alpha \eNorm{x})$ for some  $\alpha>0$.\smallskip

We define a bounded non-decreasing function $\multFunc:(0,\infty)\rightarrow [0,\infty )$ by 
\begin{eqnarray}\label{def:funck}
	\multFunc(\epsilon ) &:=& \sup \frac{\eNorm{\nabla V}}{\max\{V,1/\epsilon\}}.
\end{eqnarray}
	
In contrast to Section \ref{secGeometricSimple}, we now allow the 
constant $\epsilon$ in \eqref{multDist} to be chosen freely inside a given range. We require that
	\begin{align}\label{eq:alphaMult}
	(4C\epsilon)^{-1}\ &\geq\ \int_0^{R_1} \int_0^s  \exp\left(
	  \frac{1}{2} \int_r^s u\,\kappa(u)\, du \,+ 2\,\multFunc(\epsilon )\, (s-r)\right) dr \;
	 ds
	\end{align} 
with $C$ and $\kappa$ given by Assumptions \ref{assGeometricDrift} and \ref{assGlobalCurvBound}, respectively. Notice that since $Q$ is non-decreasing, \eqref{eq:alphaMult} is always satisfied for $\epsilon$ sufficiently small. 
Further below, we demonstrate 
how the freedom to choose $\epsilon$ can be exploited to optimize the resulting contraction rate
in certain cases. 
Similarly to Section \ref{secGeometricSimple},
we define functions 
$\phi,\Phi:\reals_+\rightarrow\reals_+$ by
\begin{equation}\label{eq:phimult}
	\phi(r)=\exp\left(\,-\frac{1}{2} \int_0^r t\, \kappa(t)\, dt\,-  2\,\multFunc(\epsilon)\, r\,\right) , \qquad\Phi(r)=\int_0^r \phi(t) \,dt.
\end{equation} 
The function $f$ in \eqref{multDist} will be chosen such that
$$\frac{1}{2}\Phi(r)\leq f(r) \leq \Phi(r)\quad\mbox{for }r\le R_2, \quad\mbox{ and }\quad f(r)=f(R_2) \quad\mbox{for }r\ge R_2,$$
where we define
\begin{eqnarray}\label{eq:S1mult}
S_1 & :=& \left\{(x,y)\in \reals^d\times\reals^d:
\lypFunc(x)+\lypFunc(y)\leq {2\lypConstant}/{\lambda}\right\},\\
\label{eq:S2mult}
S_2& :=& \left\{(x,y)\in \reals^d\times\reals^d:
\lypFunc(x)+\lypFunc(y)\leq {4\lypConstant}(1+\lambda )/\lambda\right\} ,\\
R_i  &:=& \sup\left\{ \eNorm{x-y}:(x,y)\in S_i\right\},\qquad
i=1,2.
\label{R1multsimple}
\end{eqnarray}

Here the sets $S_1$ and $S_2$ have been chosen such that 
\begin{align*}
	\generator V(x)+\generator V(y) &< 0 \qquad\text{for }(x,y)\not\in S_1, \qquad\ \text{and}\\    
	\epsilon\generator V(x)+\epsilon\generator V(y) &< - 
	\frac{\lambda}{2}\min\left\{1,4C\epsilon \right\}(1+\epsilon V(x)+ \epsilon V(y) )
	\quad\text{for } (x,y)\not\in S_2.
\end{align*}
We now state our second main result: 
\begin{thm}[Contraction rates for multiplicative semimetric]\label{thmMultMainSimply}
	Suppose that Assumptions \ref{assConstCurvBound}, \ref{assGeometricDrift}, 
	and \ref{assGeometricDriftAddOn} hold true. Fix $\epsilon\in(0,\infty)$ such that \eqref{eq:alphaMult} is satisfied.
	 Then there exist a
concave, bounded and non-decreasing continuous function $f:\reals_+\rightarrow \reals_+$
with $f(0)=0$ and a constant $c\in (0,\infty )$ such that
\begin{equation}\label{eq_contrMultSimple}
	\wDist{\rho_2}(\mu p_t,\nu p_t)\ \leq\ e^{-c\, t}\, 
	\wDist{\rho_2}(\mu,\nu)\qquad\mbox{for any }\mu,\nu\in \lypMeasures .
\end{equation} 
Here $c=\min\left\{\beta,\lambda,4C\epsilon\lambda \right\}/2$ where 
	 	\begin{align*}
 			\beta^{-1} &=\int_0^{R_2} \Phi(r)  \phi(r)^{-1}\, dr \\&=\int_0^{R_2} \int_0^s  \exp\left(
	  \frac{1}{2} \int_r^s u\,\kappa(u)\, du \,+ 2\,\multFunc(\epsilon)\, (s-r)\right) dr \;
 	 ds,
 	\end{align*}
the distance $\rho_2$ is defined by \eqref{multDist}, and $f$ is constant for $r\geq R_2$ and satisfies
	\begin{align*}
		\frac{1}{2} \ \leq\ f'(r) \, \exp\left(
	  \frac{1}{2} \int_0^r u\,\kappa(u)\, du \,+ 2\,\multFunc(\epsilon)\, r\right)\ \leq\ 1\quad\mbox{ for }r\in(0,R_2) .
	\end{align*}
		The precise definition of the function $f$ is given in the proof in Section
	\ref{secProofsMultiplicative}.  
\end{thm}

\begin{exa}[Exponential tails]
Consider again the setup of Example \ref{exa:exptails} and suppose that $\kappa\equiv 0$. Similarly as above, one verifies that in this case
$$R_1\ \le\ 2\log (4e)\, d/\delta ,\quad\mbox{and}\quad R_2\ \le\ 2\log (8e(1+\lambda ))\, d/\delta ,
\quad\mbox{where }\lambda\ =\ \delta^2/(4d).$$ 
Furthermore, $Q(\epsilon)=\alpha=\delta /d$, and therefore one can choose $\epsilon$ such that $(4C\epsilon )^{-1}$ is of order $O(d^2\delta^2)$. As a consequence, Theorem \ref{thmMultMainSimply} implies contractivity with a rate of 
order $\Omega (d^{-3})$, whereas w.r.t.\ the additive metric, we have derived a rate of order $\Omega (d^{-2}\log^{-2} (d))$ in Example \ref{assConstCurvBound}. 
\end{exa}

In order to optimize our bounds by choosing $\epsilon$ appropriately,
we replace Assumption \ref{assGeometricDriftAddOn} by a
stronger condition:

\begin{ass}\label{assGeometricDriftAddOnA}
${ \nabla V (x)}/{V(x)}\to 0$ as $|x|\to\infty$.
\end{ass}

If Assumption \ref{assGeometricDriftAddOnA} holds then 
$Q(\epsilon )\to 0$ as $\epsilon\to 0$. Therefore, by choosing
$\epsilon $ sufficiently small, we can ensure that the term
$Q(\epsilon )(s-r)$  occurring in the exponents in \eqref{eq:alphaMult} and in the
definition of $\beta$ is bounded by $1$. Explicitly, we choose 
\begin{equation}\label{eq:epschoice}
\epsilon \ =\ \min\left\{ Q^{-1}(R_2^{-1}),\, \left(4Ce^2\, I(R_1)\right)^{-1}
\right\} ,
\end{equation}
where $\,Q^{-1}(t)  := \sup \{\epsilon >0\, :\, Q(\epsilon )\le t\}  \in (0,\infty ]\,$ for $t>0$ by Assumption \ref{assGeometricDriftAddOnA}, and
\begin{eqnarray}
I(r) &:=&\int_0^{r} \int_0^s  \exp\left(
	  \frac{1}{2} \int_r^s u\,\kappa(u)\, du \right) dr \;
	 ds.\label{eq:I}
\end{eqnarray}

\begin{cor}[Contraction rates for multiplicative semimetric II]\label{corMultMainSimply} 
~\newline Suppose that Assumptions \ref{assConstCurvBound}, \ref{assGeometricDrift}, 
	and \ref{assGeometricDriftAddOnA} hold true. Then the assertion
	of Theorem \ref{thmMultMainSimply} is satisfied with $\epsilon$
	given by \eqref{eq:epschoice} and
	$$c\ \ge\ \frac 12\min\left\{ e^{-2}/I(R_2),\, \lambda ,\,\lambda e^{-2}/ I(R_1),\,4C\lambda\, Q^{-1}\left(
1/{R_2}\right)
\right\} $$
\end{cor}

The corollary is particularly useful if $b=-\nabla U$ for a convex (but not strictly convex) function $U$. In this case we can choose $\kappa =0$, and hence $I(r)=r^2/2$:

\begin{exa}[Convex case]\label{ex:convex}
	Let $b(x)=-\nabla U(x)$ for a convex function $U\in \contFunctions^2(\reals^d)$, and suppose that Assumption \ref{assGeometricDrift} holds with $V$ satisfying $V(x)=\eNorm{x}^p$ outside of a compact set for some $p\in [ 1,\infty )$. Then 
	there is a constant $A\in(0,\infty)$
	such that $Q^{-1}(t)\ge A \,t^p$ for any $t>0$, and hence
	$$c\ \ge\ \min\left\{ e^{-2}R_2^{-2},\,\lambda /2,\, \lambda e^{-2}R_1^{-2},\,2C\lambda\, A
{R_2^{-p}}
\right\} .$$  
In particular, $c^{-1}=O(R_2^2)$ if $V(x)=|x|^2$ outside a compact set.
\end{exa}

Similarly as in Corollary \ref{cor:ergodicAvergaesAdditive} above, the bounds in Theorem \ref{thmMultMainSimply}
can be used, among other things, to control the bias and variance
of ergodic averages. 
Furthermore
a statement as in (\ref{eq_contrMultSimple}) implies gradient bounds for the
transition kernel: 

\begin{cor}[Gradient bounds for the transition semigroup]\label{corMulti22} 
Suppose that the assumptions in Theorem \ref{thmMultMainSimply} are satisfied. Then
 	 $$\|{p_t g}\|_{\operatorname{Lip}(\multDist_2)}\ \leq\ e^{-c \,t}\,\|{g}\|_{\operatorname{Lip}(\multDist_2)}
	$$
	holds for any $t\ge 0$ and for any function $g:\reals^d\to\reals$ that is Lipschitz continuous w.r.t.\ $\rho_2$. In particular, if $p_tg$ is
	differentiable at $x$ then
	\begin{align}\label{gradientbound}
	\eNorm{\nabla p_t g(x)}\leq
	\| {g}\|_{\operatorname{Lip}(\multDist_2)} \,(\,1+2 \,\epsilon \,V(x)\,)\, e^{-c\,t}.\end{align}
\end{cor}

The proof is included in Section \ref{secProofsMultiplicative}
below.


\subsection{McKean-Vlasov diffusions}\label{secMckeanvlasov}

We now apply our approach to nonlinear diffusions on $\reals^d$ satisfying an SDE
of type
\begin{eqnarray}
\label{eq_mckeanvlasov} 
	dX_t &=& b(X_t) \, dt\,  +  \tau \, \int  \vartheta(X_t,y) \, \mu_t(dy) \,dt \, +\, 	dB_t, \
	X_0\, \sim\, \mu_0,\\
	\mu_t & = &\text{Law}(X_t).\nonumber
\end{eqnarray}
Here $\tau\in\reals$ is a given constant and $(B_t)$ is a $d$-dimensional Brownian motion.
Under appropriate conditions on the coefficients $b$ and $\vartheta$, Equation \eqref{eq_mckeanvlasov} has a unique solution $(X_t$) which is a nonlinear Markov process
in the sense of McKean, i.e., the future
development after time $t$ depends both on the current state $X_t$ and on the law of $X_t$ \cite{MR1108185,MR1431299}. Under Assumption 2.1 and Assumption \ref{assMcKeanVlasovLipschitz} below (where we will also assume that 
$\vartheta$ is Lipschitz), using \cite{MR1108185,MR2357669}, existence and uniqueness of the solutions hold.
Corresponding nonlinear SDEs arise
naturally as marginal limits as $n\to\infty$ of mean field interacting particle systems
\begin{align} \label{eqMeanField}
	d X_t^i\ =\ b(X_t^i)\,  dt\, +\, \frac{\tau}{n}\, \sum_{j=1}^n  \vartheta(X_t^i,X_t^j)\, dt\, +\,
	dB_t^i,\qquad i=1,\ldots n,
\end{align}
driven by independent
Brownian motions $B^i$. 

Convergence to equilibrium, or contractivity, for the nonlinear equation and the particle system are longstanding problems. Assuming $b=-\nabla V$ and
$\vartheta (x,y)=\nabla W(y)-\nabla W(x)$ with smooth potentials $V$ and
$W$, the 
convex case for the nonlinear equation was tackled by Carrillo, McCann and Villani \cite{MR2053570,MR2209130} using PDE techniques, and by
 Malrieu \cite{MR1970276} and Cattiaux, Guillin and Malrieu \cite{MR2357669} by coupling arguments. 
 More recently, using direct control of the derivative of the Wasserstein distance, Bolley, Gentil and Guillin \cite{MR3035983} have proven an exponential 
trend to equilibrium for small bounded and Lipschitz perturbations of the 
strictly convex case. In the spirit of Meyn-Tweedie's approach, and via nonlinear Markov chains, Butkovsky \cite{MR3403022} established 
exponential convergence to equilibrium in the bounded perturbation case. In
\cite[Corollary 3.4]{Eberle2015}, a contraction property for
the particle system \eqref{eqMeanField} has been derived for sufficiently small $\tau$ with a
dimension-independent contraction rate using an approximation of a componentwise reflection coupling. 
 
We now show that a similar strategy as in \cite{Eberle2015} can be applied directly to the nonlinear equation. 
We assume that the interaction coefficient $\vartheta:\reals^d\times
	\reals^d\rightarrow\reals^d$ is a globally Lipschitz continuous function:

\begin{ass}\label{assMcKeanVlasovLipschitz}
	There exists a constant $L\in (0,\infty )$ such that  
	\begin{equation*}
	\eNorm{\vartheta(x,x')-\vartheta(y,y')}\ \leq\ L\cdot \left( \eNorm{x-y} + \eNorm{x'-y'}\right)\quad\mbox{for any }x,x',y,y'\in \reals^d.
	\end{equation*}	 
\end{ass}

In our first theorem, we assume the contractivity at infinity condition
\eqref{eq:convexatinfinity} instead of a Lyapunov condition. Existence and uniqueness of solutions of the nonlinear SDE 
can then be proven as in \cite{MR2357669}.
 In that case
we can obtain contractivity w.r.t.\ an underlying metric of type
\begin{equation}\label{eq:rho0}
\rho_0 (x,y)\ =\ f(|x-y|)
\end{equation}
where $f$ is an appropriately chosen concave function. Let
$\mathcal W^1$ denote the standard $L^1$ Wasserstein distance defined
w.r.t.\ the Euclidean metric on $\reals^d$. Notice that in the next theorem, 
we allow the function 
$\kappa$ from Assumption \ref{assGlobalCurvBound} to take negative values.
We obtain the following counterpart to Corollary 3.4 in \cite{Eberle2015}:

\begin{thm}[Contraction rates for nonlinear diffusions I]\label{thmMcKeanConvexAtInfinity}
Suppose that Assumptions \ref{assGlobalCurvBound} and \ref{assMcKeanVlasovLipschitz} hold true with a function $\kappa :(0,\infty )\rightarrow
\reals$
satisfying
\begin{equation}\label{eq:convexatinfinity} 
\limsup_{r\rightarrow \infty} \kappa (r) \, < \, 0.
\end{equation}
For probability measures $\mu_0$ and $\nu_0$ with finite second moments, let $\mu_t$, resp. $\nu_t$ ($t\ge 0$) denote the marginal laws of a strong solution $(X_t)$ of Equation \eqref{eq_mckeanvlasov} with initial condition $X_0\sim\mu_0$, resp. $X_0\sim\nu_0$.
Then there exist a
concave and non-decreasing continuous function $f:\reals_+\rightarrow \reals_+$
with $f(0)=0$ and constants $c,K,A\in (0,\infty )$ such that
 for any
	$\tau\in \reals$ and initial laws $\mu_0 ,\nu_0$ with finite second moments, 
	\begin{eqnarray}
	\wDist{\rho_0}(\mu_t,\nu_t)&\leq & \exp\left(\, (\eNorm{\tau} K-c) \, t\right)\,
		\wDist{\rho_0}(\mu_0,\nu_0),\qquad\mbox{and}\label{mckeancontr1}\\
	\label{mckeancontr2}	\mathcal W^1(\mu_t,\nu_t)&\leq &
	2\,A\, \exp\left(\,(\eNorm{\tau}
	K-c)\, t\,\right)\, 	\mathcal W^1(\mu_0,\nu_0).
	\end{eqnarray}		
The constants are explicitly given by  
	\begin{eqnarray}\nonumber
	c^{-1} &=&\int_0^{R_2} \int_0^s \exp\left(\,\frac{1}{2}\,\int_r^s u\,\kappa^+(u),\,
		du\,\right) dr \,ds, 
		\\\nonumber A &=&  \exp\left(\,\frac{1}{2}\,\int_0^{R_1} s\,\kappa^+(s)\, ds\,\right)\, ,
		\\K&=&4\, L \,\exp\left(\,\frac{1}{2}\,\int_0^{R_1} \,s\kappa^+(s)\, ds\,\right), \label{mckean1K}
	\end{eqnarray}
	where
	\begin{eqnarray}
		R_1& =&\inf\{R\geq 0: \kappa(r) \leq 0 \; \text{ for all } r\geq R\},\qquad\mbox{and}\label{eq:defR1}\\
		R_2&=&\inf\{R\geq R_1: \kappa(r)\, R\,(R-R_1) \leq -4 \; \text{ for all } r\geq
		R\}.\label{eq:defR2}
	\end{eqnarray}	
	The function $f$ is linear for $r\geq R_2$, and 
	\begin{align*}
		\frac{1}{2}&\leq f'(r)\, \exp\left(\,\frac{1}{2}\,\int_0^{r\wedge
		R_1} s\,\kappa^+(s)\, ds\right) \leq 1\qquad\mbox{for }0<r< R_2.
	\end{align*}
	The precise definition of the function f is given in the proof in Section
	\ref{proofMcKeanVlasov}. 
\end{thm}

Our next goal is to replace \eqref{eq:convexatinfinity} by the following
dissipativity condition:

 \begin{ass}[Drift condition]\label{assMcKeanDrift}
	There exist constants
	$D,\lambda\in (0,\infty )$ such that
	$$\esProd{x}{b(x)}\ \leq\ - \lambda \eNorm{x}^2\qquad \mbox{for any } x\in\reals^d\mbox{ with }\eNorm{x}\geq D.$$
\end{ass}

Let $V(x):=1+\eNorm{x}^2$. Assumption
\ref{assMcKeanDrift} implies that $V$ is a Lyapunov function
for the nonlinear diffusion \eqref{eq_mckeanvlasov}, cf.\
Lemma \ref{lemMcKeanApriori} below.

A major difficulty in the McKean-Vlasov case is that solutions $X_t$ and $Y_t$ with different
initial laws follow dynamics with different
drifts. Therefore, it is not clear how to construct a coupling $(X_t,Y_t)$ such 
that $X_t=Y_t$ for $t>T$ holds for an almost surely finite stopping time $T$.
Using the multiplicative semimetric 
we are still able to
retrieve a local contraction:

\begin{thm}[Contraction rates for nonlinear diffusions II]\label{thmMcKeanDrift}
Suppose that Assumptions \ref{assGlobalCurvBound}, \ref{assMcKeanVlasovLipschitz} and \ref{assMcKeanDrift} hold true.
For probability measures $\mu_0$ and $\nu_0$ with finite second moments, let $\mu_t$, resp. $\nu_t$ ($t\ge 0$) denote the marginal laws of a strong solution $(X_t)$ of Equation \eqref{eq_mckeanvlasov} with initial condition $X_0\sim\mu_0$, resp. $X_0\sim\nu_0$.
Then there exist a
concave, bounded and non-decreasing continuous function $f:\reals_+\rightarrow \reals_+$
with $f(0)=0$ and constants $c,\epsilon, K_0,K_1 \in (0,\infty )$ s.t.:
\begin{itemize}
  \item[(i)] For any $R\in (0,\infty )$ there is $\tau_0\in(0,\infty )$ such that for any
	$\tau\in \reals$ with $\eNorm{\tau}\leq \tau_0$, and initial laws with
	$\mu_0(V),\nu_0(V)\le R$,  
	\begin{align}
	\wDist{\rho_2}(\mu_t,\nu_t)&\leq \exp(-c \,t)\,
		\wDist{\rho_2}(\mu_0,\nu_0), \qquad \text{and}\;\label{mckeancontr3}\\
		\mathcal W^{1}(\mu_t,\nu_t)&\leq K_0\, \exp(-c \, t)\, \wDist{\rho_2}(\mu_0,\nu_0) \;.\label{mckeancontr4}
	\end{align}
	\item[(ii)] There is $\tau_0\in(0,\infty)$ s.t.\ for any
	$\tau\in \reals$  with $\eNorm{\tau}\leq \tau_0$ and initial laws $\mu_0,\nu_0$ with finite second moment,	
	 \begin{align}
	\wDist{\rho_2}(\mu_t,\nu_t)&\leq \exp(-c \,t)\,
		\left(\wDist{\rho_2}(\mu_0,\nu_0)+K_1\,[\epsilon \mu(V)+\epsilon \nu(V)]^2\right).\label{mckeancontr5}
	\end{align}
\end{itemize}
The function $\rho_2$ is given by \eqref{multDist}.
	For the explicit definition of the function $f$ and the constants $c,\epsilon, \tau_0, K_0, K_1$ see the
	proof in Section \ref{proofMcKeanVlasov}.
\end{thm}

The
assumption that $\tau$ is sufficiently small is natural, 
since for large $\tau$, Equation \eqref{eq_mckeanvlasov} can have several distinct stationary solutions. It is implicit here that due to the contractions, uniqueness of the invariant measure holds. 
Nevertheless, we do not claim that our bound on $\tau$ is sharp.


\subsection{Subgeometric ergodicity with explicit constants} \label{secSubGeometricSimple}
We now consider the case where the drift is not strong enough to provide a Kantorovich contraction
like \eqref{eq_contrAddSimple}. Instead of  Assumption \ref{assGeometricDrift} we only assume a subgeometric drift condition as it has been used for example in \cite{MR2499863}.

\begin{ass}[Subgeometric Drift Condition]\label{assSubDrift}
There are a
function $\lypFunc\in \contFunctions^2(\reals^d)$ with $\inf_{x\in\reals^d} V(x)>0$, a strictly positive, increasing and concave $\contFunctions^1$ function $\concFunc:\reals_+\rightarrow \reals_+$ such that $\eta(\,\lypFunc(x)\,)\rightarrow \infty$ as $\eNorm{x}\rightarrow\infty$, 
as well as a constant $C\in (0,\infty )$ such that 
\begin{equation}\label{eq:LyapunovSubGeometric}
\generator \lypFunc(x)\ \leq\ \lypConstant- \concFunc(\,\lypFunc(x)\, )\qquad
\mbox{for any }x\in\reals^d. 
\end{equation}   
\end{ass} 
The following example shows how $V$ and $\eta$ can be chosen explicitly, cf.\ also \cite{MR2499863}.
\begin{exa}[Choice of $V$ and $\eta$] \label{exaSubGeometricLyapunov}
Suppose that
\begin{align}
		\esProd{b(x)}{x}\leq - \gamma \eNorm{x}^q \qquad\mbox{for } \eNorm{x}\geq R \end{align}
holds with constants $R,\gamma\in (0,\infty )$ and $q\in (0,1)$.
Let $V\in C^2(\reals^d)$ be a strictly positive function such that outside a compact set, 
$V(x)=\exp(\alpha
	\eNorm{x}^q)$  for some $\alpha\in (0,{2\gamma}/{q})$,
	and fix 
	$\beta\in(0,\gamma-{\alpha q}/{2})$.
Then Assumption \ref{assSubDrift} is satisfied with	
$$\eta(r) = \left\{
\begin{array}{ll}
\alpha^{\frac{2}{q}-1} q \beta  r \log(r)^{2-\frac{2}{q}}
&\mbox{for }r\geq e^{\frac{2}{q}-1 },\\
\alpha^{\frac{2}{q}-1}  \beta  \left(\frac{2}{q}-1\right)^{1-\frac{2}{q}} 
	    \left(2 e^{1-\frac{2}{q}}
		(q-1) r^2 + (4-3 q) r\right) 
&\mbox{for }r<e^{\frac{2}{q}-1}.
\end{array}
\right.$$
\end{exa}

From now on we assume that Assumption \ref{assGlobalCurvBound} holds true, and we define the functions $\varphi$ and $\Phi$ as
in \eqref{eq:phi} above. Let
$R_1:=\sup\left\{ \eNorm{x-y}:(x,y)\in S_1\right\}$, where
\begin{equation}\label{eq:S1subgeometric}
S_1\ :=\ \left\{(x,y)\in \reals^d\times\reals^d:
\concFunc(\,\lypFunc(x)\,)+\concFunc(\,\lypFunc(y)\,)\leq 4\lypConstant \right\} .
\end{equation}
The set $S_1$ is chosen such that for $(x,y)\not\in S_1$, 
\begin{equation}\label{eq:Lyapunovsubgeoetric}
\generator\,	\lypFunc(x)+\generator \,\lypFunc(y)\leq - \, \left(\, \concFunc(\,\lypFunc(x)\,)+\concFunc(\,\lypFunc(y)\,)\, \right) /2.
\end{equation}
Notice that since $\eta(\,V(x)\,)\rightarrow \infty$ as $\eNorm{x}\rightarrow\infty$, $R_1$ is finite, and $S_1$ is recurrent for any Markovian coupling $(X_t,Y_t)$ of solutions of \eqref{eqStdDiffusion}. Let
	\begin{equation}\label{eq:epsilonsubgeometric}
	\epsilon^{-1}\,=\, \max\left(1,4C\,\int_0^{R_1} \phi(r)^{-1} \, dr \right) 
	\,=\, \max\left( 1,4C\, \int_0^{R_1} e^{\frac{1}{2} \int_0^r t\, \kappa(t)\, dt} \, dr \right).
	\end{equation}

The following growth condition on the Lyapunov function replaces Assumption~\ref{assAddGrowthCurvBound}:

\begin{ass}[Growth condition in subgeometric case]\label{assAddGrowthCurvBoundSubgeometric}
There exist a constant $\alpha>0$ and a bounded set $S_2\supseteq S_1$ such that for any
	$(x,y)\in  \reals^{2d}\setminus S_2$,
	 \begin{eqnarray}
	\lefteqn{	 	 \concFunc(\,V(x)\,)+\concFunc(\,V(y)\,)}
 \label{eq:smplAddSetS2subgeometric}\\
 \nonumber &\geq &
	 	\max\left( 4C, 1/\int_0^{R_1}  \phi(r)^{-1} \, dr\right)\, \left(1+\alpha \, \int_0^{R_1}  \phi(r)^{-1} \, dr \; \concFunc(\,\Phi(\eNorm{x-y})\,)\,\right).
\end{eqnarray}	   
\end{ass} 
Notice that $\Phi$ grows at most linearly. Let
$
R_2 := \sup\left\{ \eNorm{x-y}:(x,y)\in S_2\right\}$. We state our main result for the subgeometric case.  

\begin{thm}[Subgeometric decay rates]\label{thmMainSubgeometric}
Suppose that Assumptions  \ref{assGlobalCurvBound}, \ref{assSubDrift}  and 
\ref{assAddGrowthCurvBoundSubgeometric} hold true. 
Then there exist a
concave, bounded and non-decreasing continuous function $f:\reals_+\rightarrow \reals_+$
with $f(0)=0$ and constants $c,\epsilon\in (0,\infty )$ s.t.\ 
	\begin{align}\label{eqSubMain}
\|		p_t(x,\cdot)-  p_t(y,\cdot)\|_{TV}\  \leq \
		\frac{\rho_{1}(x,y)}{H^{-1}(c\; t)}
\quad \text{ for any }		x,y\in\reals^d\mbox{ and }t\ge 0.
 	\end{align}
 	Here the distance $\rho_1$ is defined by \eqref{addDistance} 
and \eqref{eq:epsilonsubgeometric}, the function 
$H:[l,\infty)\rightarrow [0,\infty)$
	is given by
	 \begin{align}
	 	H(t):=\int_l^t \frac{1}{\concFunc(s)} \,ds \qquad\text{with}\qquad l=2\,\epsilon\,\inf_{x\in\reals^d} V(x),
	\end{align}
	$c=\min\left\{\alpha,\beta,\gamma\right\}/2$ where $\beta$ is given by \eqref{defbeta}, and
$$		\gamma\ =\ \inf\left\{\,{\epsilon\,\eta(r)}/{\eta(\epsilon r)}: \, r\geq {l}/{\epsilon} \right\}.
$$
	The function $f$ is constant for $r\geq R_2$, and
	\begin{align*}
		\frac{1}{2} \ \leq\ f'(r) \, \exp\left(\,\frac{1}{2} \int_0^r t\, \kappa(t)\, dt\,\right)\ \leq\ 1\qquad\mbox{ for any }r\in(0,R_2) .
	\end{align*}
		The precise definition of the function $f$ is given in the proof in Section
	\ref{secProofsSubGeometricSimple}.   
\end{thm}

The crucial difference in comparison to Theorem \ref{thmAddMainSimple} is that we do not provide upper bounds on $\wDist{\addDist_1}$, 
but use the 
additive metric to derive moment bounds for coupling times instead. These bounds are partially based
on a technique from \cite{hairerlecturenotes}, see Section \ref{dis:subgeometric} further below.

\begin{rem}
	Since $\eta (s)$ is concave,
it is growing at most linearly as $s\to\infty$. In particular,
$\int_l^\infty (1/\eta (s))\, ds\,=\,\infty$, and thus
	the inverse function $H^{-1}$ maps $[0,\infty)$ to $ [l,\infty)$. Since $\epsilon\leq 1$ and $\eta$ is increasing, we always have
	$\gamma \geq \epsilon$ . If $\eta(r)=r^a$ for some $a\in (0,1)$ then
	$\gamma= \epsilon^{1-a}$.
\end{rem}

It is well-known that the local Lipschitz assumption on $b$ together with Assumption \ref{assSubDrift} implies
the existence of a unique invariant probability measure $\pi$ satisfying
$\int \concFunc(\,V(x)\,) \, \pi(dx) \leq C$, see e.g. \cite[Section 4]{hairerlecturenotes}.
Theorem \ref{thmMainSubgeometric} can be used to quantify the speed of convergence towards 
the invariant measure using cut-off arguments. Following 
\cite[Section 4]{hairerlecturenotes}, we obtain:

\begin{cor}\label{corsub4}
	Under the Assumptions of Theorem \ref{thmMainSubgeometric}, 
$$	\gNorm{ p_t(x,\cdot ) - \pi}_{TV} \leq \frac{R_2+\epsilon\, V(x)}{H^{-1}(\,c\,
t\,)}+\frac{(2\,\epsilon \,b+1) \,C
}{\concFunc(\,b\, H^{-1}(\,c\, t\,)\,)} \ \text{ for any } x\in\reals^d\mbox{ and }
t\ge 0,
	$$
where $b:=\eta^{-1}(\,2\,C\,)/{l}$.
\end{cor}
The proofs are given in Section \ref{secProofsSubGeometricSimple}.

\section{Discussion}\label{secDiscussion}

\subsection{Comparison to Meyn-Tweedie approach}

The classical Harris theorem, as propagated by Meyn-Tweedie, allows to derive geometric ergodicity for a large class of 
Markov chains under conditions which are easy to verify. 
The approach is very generally applicable, but it is usually not trivial to make the results quantitative.
The first assumption is 
that the Markov  chain
at hand is recurrent w.r.t.\ some
bounded subset $S$ of the state space and that one has some kind of control over the average length of excursions from this set. 
The second assumption which is typically imposed is a minorization condition which often takes the following 
form: There are constants $t,\epsilon\in(0,\infty)$
and a probabilty measure $Q$ such that 
\begin{equation}\label{eq:minorization}
p_t(x,\cdot )\geq \epsilon Q(\cdot)
\end{equation}
holds for all $x\in S$, where $p_t$ denotes the transition kernel of the chain. 

The recurrence condition can be quantified performing direct  
computations with the generator of the Markov chain via Lyapunov techniques. 
The minorization condition is usually much harder to quantify. In the context
of diffusions of the form \eqref{eqStdDiffusion} there are abstract methods available
which allow to conclude that 
the condition \eqref{eq:minorization} can indeed be satisfied, cf.\ \cite[Remark 1.29]{kulik2015introduction}.
Nevertheless, using such methods, it is not clear how the resulting constant $\epsilon$ depends 
on the drift coefficient $b$,
and how a perturbation of $b$ translates to a change of $\epsilon$. In the diffusion setting,  
Roberts and Rosenthal developed in \cite{MR1423462} a method to provide explicit bounds for $\epsilon$ 
that are closely connected to the drift coefficient $b$. Their method
is based on reflection coupling and an application of the Bachelier-L\'evy formula. In comparison to their results, 
we establish contractions of the transition kernels, and our contraction rates are based only on {\em one-sided} Lipschitz bounds for the drift coefficient. This often
leads to much more precise bounds.

\subsection{Relation to functional inequalities} 
Functional inequalities are now a common tool to get rates for convergence to equilibrium in $L^2$ distance or in entropy. For the class of diffusion processes considered here, the Poincar\'e inequality takes the form
\begin{equation}\label{PI}
\mbox{Var}_\pi(f)\ \le\ \frac{1}{2} C_P\int|\nabla f|^2d\pi
\end{equation}
for smooth functions $f$, where $\pi$ is the stationary distribution. \eqref{PI} is equivalent to $L^2$ convergence to equilibrium (and in fact $L^2$ contractivity) with rate $C_P^{-1}$. It turns out to be quite difficult to prove a Poincar\'e inequality for a general non-reversible diffusion such as \eqref{eqStdDiffusion}, as usual criteria rely on the explicit knowledge of the invariant probability measure $\pi$. If we assume that $b(x)=-\nabla V(x)/2$ then the diffusion is reversible with respect to $d\pi=e^{-V}dx$ and plenty of criteria are available to prove Poincar\'e inequalities. In particular, it is shown in \cite{MR2386063}, that if there exists a set $B$, 
constants $\lambda,C\in (0,\infty )$, and a positive twice continuously differentiable function $V$ such that
$${\mathcal L}V\ \le -\lambda V\,+\,C1_B,$$
and a local Poincar\'e inequality of the form
$$\int_B(f-\pi(f1_B))^2d\pi\ \le\ \frac 12 \kappa_B\int|\nabla f|^2d\pi$$
holds, then a global Poincar\'e inequality holds with constant $C_P=\lambda^{-1}(1+C\kappa_B)$. 
Note that a Poincar\'e inequality implies back the Lyapunov condition. Using the additive metric and Corollary \ref{corAddSimple2}, 
one has that a Poincar\'e inequality holds but the identification of the constant is a hard task in general. However, using the multiplicative metric and 
the gradient bounds of Corollary \ref{corMulti22}, one may prove that  in the reversible case a Poincar\'e inequality holds with 
the same constant $c$ than in Corollary \ref{corMulti22}. 
Here the reflection coupling serves as an alternative to a local Poincar\'e inequality. The latter is
usually established via Holley-Stroock's perturbation argument which may lead to quite poor estimates.\smallskip

Notice also that, by a result of Sturm and von Renesse \cite{MR2142879}, for a reversible diffusion with stationary 
ditribution $e^{-V}\,dx$,
a strict contraction in $L^p$ Wasserstein distance
is {\em equivalent} to a lower bound on the Hessian of $V$. The latter condition is a special case of the Bakry-Emery criterion and usually linked to logarithmic Sobolev inequalities. In
\cite{MR2498560}, a reinforced Lyapunov condition has been used to prove stronger functional inequalities than Poincar\'e inequalities
(namely super Poincar\'e inequalities, including logarithmic Sobolev inequalities).
In a similar spirit, we are now able to remove the global curvature condition assuming a reinforced Lyapunov condition.
Note however, that although our 
results are sufficient to prove back some Poincar\'e inequality, it does not seem possible to get stronger inequalities 
starting from our contractions.

\subsection{Dimension dependence}
In our results above, dependence on the dimension $d$ usually enters through
the value of the constant $C$ in the Lyapunov condition, which affects
the size of $R_2$. For example, in 
Theorem \ref{thmAddMainSimple}, the contraction rate is
$c=\min\left\{\alpha, \beta,\lambda\right\}/2$, where $\alpha$ and $\lambda$ are given by Assumptions 
\ref{assAddGrowthCurvBound} and \ref{assGeometricDrift} respectively, and the constant $\beta$ defined in \eqref{defbeta}
depends both on $R_2$ and on the function $\kappa$ in Assumption \ref{assConstCurvBound}. 
In order to illustrate the dependence on the dimension
of $R_2$, let us assume 
that there are constants $A,\gamma\in(0,\infty)$ and
$q\geq 1$ such that
\begin{align}\label{examplequation1}
	\esProd{x}{b(x)}\leq - \gamma\eNorm{x}^q \qquad\text{for all } \eNorm{x}\geq A.
\end{align}
Suppose first that $q=2$. Then $V(x)=1+|x|^2$ satisfies the Lyapunov condition in Assumption \ref{assGeometricDrift} with constants $C=O(d)$ and $\lambda =\Omega (1)$. In this case, the set $S_2$ in Assumption 
\ref{assAddGrowthCurvBound} can be chosen such that $R_2=O(\sqrt d)$.
Hence assuming a one-sided Lipschitz condition with constant $\kappa$ as in Example \ref{assConstCurvBound}, the lower bound $c$ for the contraction rate in Theorem \ref{thmAddMainSimple} is of order $\Omega (1/d)$
if $\kappa =0$ (convex case), or, more generally, if $\kappa =O(1/d)$. On the other hand, for
$\kappa =\Omega (1)$, $c$ is exponentially small in the dimension.
By Example \ref{ex:convex}, similar statements hold true for the lower
bound on the contraction rate w.r.t.\ the multiplicative semimetric
derived in Corollary \ref{corMultMainSimply}.\smallskip

Now assume more generally $q\ge 1$. In this case, a Lyapunov function with polynomial growth does not necessarily exist. Instead, by Remark
\ref{additive_example_lyapunov},  one can choose a Lyapunov function $V$
with constant $\lambda=1$ such that outside of a compact set,
$V(x)=\exp\left(a \eNorm{x}^q\right)$ for some $a<2\gamma/q$.
In this case, $C=O(\exp (\eta d))$ for some finite constant $\eta >0$, and
one can choose $R_2$ of order $O(d^{1/q})$. Again, assuming a 
one-sided Lipschitz condition, the constant $c$ in Theorem \ref{thmAddMainSimple} is of polynomial order $\Omega (d^{-2/q})$
if $\kappa =0$ (convex case), or, more generally, if $\kappa =O(d^{-2/q})$. For the multiplicative semimetric, we are not able to prove a 
polynomial order in the dimension in this case; for $q\in (1,2)$, an
application of
Corollary \ref{corMultMainSimply} with a Lyapunov function satisfying $V(x)=\exp (|x|^\alpha )$ for large $|x|$ for some $\alpha\in (2-q,1)$ at least yields a sub-exponential
order in $d$. For $\kappa =\Omega (1)$, the values of $c$ decay exponentially in the dimension.\smallskip

We finally remark that in some situations it is possible to combine the techniques presented here with additional arguments to derive explicit and dimension-free
contraction rates for diffusions, see for example \cite{zimmer16}.

\subsection{Extensions of the results}

Similarly as in \cite{Eberle2015}, the results presented above can be easily generalized to diffusions
with a constant and non-degenerate diffusion matrix $\sigma$. In the case of non-constant and non-degener\-ate diffusion coefficients $\sigma(x)$,
it should still be possible to retrieve related results replacing reflection coupling by the \emph{Kendall-Cranston coupling} w.r.t.\
the intrinsic Riemannian metric induced by the diffusion coefficients.

The main contraction results, Theorem \ref{thmAddMainSimple} and Theorem \ref{thmMultMainSimply},
are based on Assumption \ref{assGlobalCurvBound}, a \emph{global} generalized one-sided Lipschitz condition.
It is possible to relax this condition to a \emph{local} bound
which, up to some technical details, holds only on the set for which the coupling $(X_t,Y_t)$ is recurrent. 
A corresponding generalization is given in \cite{Zimmer}.

In the recent work \cite{2015arXiv150908816M}, Majka extends
the results from \cite{Eberle2015} to stochastic differential equations driven by L\'evy jump processes
with rotationally invariant jump measures, thus deriving Kantorovich contractions for the transition
semigroups with explicit constants, see also \cite{JWang}. One of the key assumptions in \cite{2015arXiv150908816M} is 
the ``contractivity at infinity'' condition \eqref{eq:convexatinfinity}. Using an additive distance similar to \eqref{addDistance}, it should be possible to extend the results presented there, replacing the latter
assumption by a more general geometric drift condition.

An extension of the theory presented in this paper to a class of degenerate and infinite-dimensional diffusions is considered in
\cite{zimmer16} combining asymptotic couplings with
the multiplicative distance \eqref{multDist}.

In this work, we derive explicit contraction rates for diffusion processes. An important question is whether
similar results can be obtained for time-discrete approximations. There are at least two different approaches to tackle
this question. The first approach, which is considered
in forthcoming work by one of the authors, is to establish related coupling approaches directly for Markov chains. Another possibility 
is to consider time discretizations as a perturbation of the diffusion
process, and to apply directly the contraction results for the diffusion,
cf.\ \cite{Dalalyan, DurmusMoulines, DurmusMoulinesB} and also
 \cite{MR1622440,MR1756418,MR2012680,2014arXiv1405.0182P,MR3076780,
 2015arXiv150304123R}.

\subsection{McKean-Vlasov equations}

For the class of nonlinear diffusions considered above, 
Theorems \ref{thmMcKeanConvexAtInfinity} and \ref{thmMcKeanDrift}
above considerably relax assumptions in previous works. Both the PDE
approach in \cite{MR2053570} and the approach based on synchronuous coupling in \cite{MR1970276,MR2357669} require global 
positive curvature bounds. In the case where the curvature is strictly positive with degeneracy at a finite number of points, algebraic contraction rates have been derived by synchronuous  coupling.
The dissipation of $W^2$ approach in \cite{MR3035983} yields exponential
decay to equilibrium for sufficiently small $\tau$ provided the confinement and interaction forces both derive from a potential, 
the confinement force satisfies condition \eqref{eq:convexatinfinity}, and the interaction potential is bounded with a lower bound on the curvature.
The approach in \cite{MR3403022} yields exponential convergence to equilibrium in total variation distance
in the small and bounded interaction case. 
Theorem \ref{thmMcKeanConvexAtInfinity} above relaxes these assumptions on the interaction potential while requiring only a
``strict convexity at infinity'' condition on the confinement potential. Moreover,
Theorem \ref{thmMcKeanDrift} replaces the latter condition on the confinement 
potential by the dissipativity condition in Assumption  \ref{assMcKeanDrift}. With additional technicalities, it should be possible to relax this dissipativity condition to $\esProd{x}{b(x)}\leq - \lambda\,\eNorm{x}$.
 
\subsection{Subgeometric ergodicity}\label{dis:subgeometric}
Our results in the subgeometric case can be interpreted as a variation of statements from the lecture notes \cite[Section 4]{hairerlecturenotes}.
There, M. Hairer derives subgeometric ergodicity for diffusions, estimating hitting times of recurrent sets and combining
these with a minorization condition. While the principle result from \cite[Section 4]{hairerlecturenotes} is already contained
in \cite{MR2381160,MR2499863}, the method of proof shows new and interesting aspects avoiding discrete-time approximations. 
The main tool used is the following statement, 
which gives an elegant proof for the integrability of hitting times:
\begin{lem}[\cite{hairerlecturenotes}, reformulated]\label{subgeometricKnownResult}
Let $\concFunc:\reals_+\rightarrow \reals_+$ be a strictly positive, increasing and concave $\contFunctions^1$ 
function and denote by $(Z_t)$  a continuous semimartingale, i.e.\, $Z_t=Z_0+A_t+M_t$,  
where $(A_t)$ is of finite variation, $(M_t)$ is a local martingale, $E[Z_0]<\infty$ and $A_0=M_0=0$.
Let $T$ be a stopping time. If there are constants $l,c\in(0,\infty)$ such that
\begin{align}\label{lemsubass}
	 &Z_t\geq l &  &\text{and} & &dA_{t}\leq -c \, \eta(\,Z_t\,)\,
	 ds \quad\text{almost surely for $t<T$},
\end{align} 
then $T$ is almost surely finite and satisfies the inequality
	\begin{align*}
		E\left[\, H^{-1}(\,c\, T\,)\,\right]&\leq E\left[\,H^{-1}(\,H(\,Z_{T}\,) +
		c \,T\, )\right] \leq E[\,Z_0\,],
	\end{align*}
	where $H:[l,\infty)\rightarrow [0,\infty)$ is given by  
	$H(t):=\int_l^t \frac{1}{\concFunc(s)} \,ds$.
\end{lem}
Our result for the subgeometric case, Theorem \ref{thmMainSubgeometric}, relies on the above tool. The main difference 
to \cite{hairerlecturenotes} is that we do not impose any kind of minorization condition or renewal theory.  
 Instead we 
consider a reflection coupling  $(X_t,Y_t)$ of the diffusions, 
defined in Section \ref{ReflectionCoupling},
and we directly establish bounds on the integrability of the coupling time $T:=\inf\{t\geq 0: X_t=Y_t\}$  using Lemma \ref{subgeometricKnownResult}
and the additive distance \eqref{addDistance}.
For the reader's convenience, a proof of Lemma \ref{subgeometricKnownResult} is included in Section \ref{secProofsSubGeometricSimple}.
It should be mentioned that 
subgeometric ergodicity of Markov processes has been studied by many others authors in various settings, 
see \cite{MR2071426,MR2134115,MR711187,MR1285459,MR1472961,MR3178490,MR1967786,MR2157514}
and the references therein.

\section{Couplings} \label{secCoupling}
\subsection{Synchronuous coupling for diffusions}\label{SynchronuousCoupling}
Given initial values $(x_0,y_0)\in\reals^{2d}$ and a $d$-dimensional Brownian motion $(B_t)$, we define a \emph{synchronuous coupling} 
of two solutions of \eqref{eqStdDiffusion}
as a diffusion process  $(X_t,Y_t)$ with values in $\reals^{2d}$ solving
\begin{align*}
	dX_t &= b(X_t) \, dt +  \,dB_t, & X_0&=x_0,	\\
	dY_t &= b(Y_t) \, dt +  \,dB_t, & Y_0&=y_0.
\end{align*}
 
\subsection{Reflection coupling for diffusions}\label{ReflectionCoupling}
Reflection  coupling goes back to \cite{MR841588}, where existence and uniqueness of strong solution is proved for the associated diffusion processes. Given initial values $(x_0,y_0)\in\reals^{2d}$ and a $d$-dimensional Brownian motion $(B_t)$, a \emph{reflection coupling} of two solutions of \eqref{eqStdDiffusion}
as a diffusion process  $(X_t,Y_t)$ with values in $\reals^{2d}$ satisfying
\begin{align*}
	dX_t &= b(X_t) \, dt +  dB_t, \qquad	(X_0,Y_0)=(x_0,y_0),\\
	dY_t &= b(Y_t) \, dt +  (I-2\,e_t\esProd{e_t}{\cdot}) \, dB_t \quad  \text{ for }t<T,\qquad
	Y_t  = X_t \quad \text{ for }t\geq T, 
\end{align*} 
where $T=\inf\{t\geq 0: X_t=Y_t\}$ is the coupling time. Here, for $t<T$, $e_t$ is the unit vector given by $e_t={(X_t-Y_t)}/{\eNorm{X_t-Y_t}}$. 
 
\subsection{Coupling for McKean-Vlasov processes}\label{sec:couplingMcKean}
We construct a coupling for two solutions of \eqref{eq_mckeanvlasov}. 
The coupling will be realized as a process $(X_t,Y_t)$ with values in $\reals^{2d}$.
We first describe the coupling in words: We fix a parameter $\delta>0$ and use a reflection coupling of the driving Brownian motions
whenever $\eNorm{X_t-Y_t}\geq \delta$. If, on the other hand, $\eNorm{X_t-Y_t}\leq {\delta}/{2}$ we use a synchronuous coupling.
Inbetween there is a transition region, where a mixture of both couplings is used. One should think of $\delta$ being close to zero.
\par
The technical realization of the coupling is near to \cite{Eberle2015}. In order to implement the above coupling, we introduce 
Lipschitz functions $\operatorname{rc}: \reals^d\times \reals^d\rightarrow[0,1]$ and
 $\operatorname{sc}: \reals^d\times \reals^d\rightarrow[0,1]$ satisfying
 \begin{align}
 	\operatorname{sc}^2(x,y)+ \operatorname{rc}^2(x,y) &= 1.  \label{levycondition}
 \end{align}
 We impose that $\operatorname{rc}(x,y)=1$ holds whenever $\eNorm{x-y}\geq \delta$ and $\operatorname{rc}(x,y)=0$ holds if $\eNorm{x-y}\leq {\delta}/{2}$.
 The functions $\operatorname{rc}$ and $\operatorname{sc}$ can be constructed using standard cut-off techniques.
Notice that in the case where the drift coefficient $b$ and the nonlinearity $\vartheta$ are Lipschitz,  equation \eqref{eq_mckeanvlasov} admits
a unique, strong and non-explosive solution $(X_t)$ for any initial probability measure $\mu_0$, for which we always assume finite second moment. The uniqueness holds
pathwise and in law. Moreover, the
law $\mu_t$ of $X_t$ has finite second moments, i.e. $\int \eNorm{y}^2 \mu_t(dy)<\infty$, 
see \cite[Theorem 2.2]{MR1431299} and \cite{MR1108185}. For a fixed initial probability measure $\mu_0$ we define
$$b^{\mu_0}(t,y)\ =\ b(y)+\tau \int \vartheta(y,z) \, \mu_t(dz).$$ 
The results from \cite[Theorem 2.2]{MR1431299} imply that the 
function $b^{\mu_0}:\reals_+\times\reals^d\rightarrow\reals^d$ is continuous. It is easy to see that 
 Assumption \ref{assMcKeanVlasovLipschitz}, in combination with a Lipschitz bound on $b$, implies that
there is $M>0$ such that
\begin{align} \label{lemCouplMcKeanVlasov}
	\sup_{t\geq 0}  \eNorm{\, b^{\mu_0}(t,y)-b^{\mu_0}(t,z)\, } \leq M \cdot \eNorm{y-z} \qquad\text{for any $y,z\in\reals^d$}. 
\end{align}
Fix now initial probabiliyu measures $\mu_0$ and $\nu_0$, the parameter $\delta>0$ and two independent Brownian motions $(B_t^1)$ and $(B_t^2)$.
For the given $\mu_0$ and $\nu_0$ we construct drift coefficients $b^{\mu_0}$ and $b^{\nu_0}$ as above and
define the coupling $(U_t)=(X_t,Y_t)$ as the solution of the standard diffusion 
	 \begin{align*}
 	dX_t &= b^{\mu_0}(t,X_t) \, dt + \operatorname{rc}(U_t) \; dB^1_t +
 		\operatorname{sc}(U_t)   \; dB^2_t \\
 	dY_t &= b^{\nu_0}(t,Y_t) \, dt +  \operatorname{rc}(U_t) \; (I-2e_t
  		\esProd{e_t}{\cdot}) \; dB^1_t + \operatorname{sc}(U_t) \; dB^2_t,
 \end{align*}
with $(X_0,Y_0)=(x_0,y_0)$ and $$e_t\ =\
 \frac{X_t-Y_t}{\eNorm{X_t-Y_t}}\quad\mbox{for }
 X_t\not=Y_t,\qquad e_t\ =\  u \quad\mbox{for } X_t=Y_t,  $$
 where $u\in\reals^d$ is some arbitrary fixed unit vector. 
 Note that the concrete choice of $u$ is irrelevant for the dynamic, since $\operatorname{rc}(x,x)=0$. 
 Inequality \eqref{lemCouplMcKeanVlasov} implies that the above diffusion process admits a unique, 
 strong and non-explosive solution. 
Using Levy's characterization of Brownian motion and
\eqref{levycondition}, one can verify that the marginal processes $(X_t)$ and $(Y_t)$
solve the standard equations
\begin{align}\label{eq:CouplMcKeanVlasov1}
	dX_t = b^{\mu_0}(t,X_t) \, dt + dB_t, & & X_0=x_0 \\
	dY_t = b^{\nu_0}(t,Y_t) \, dt + d\hat{B}_t, & & Y_0=y_0.\label{eq:CouplMcKeanVlasov2}
\end{align}
with respect to the Brownian motions 
\begin{align}\label{BMcoupling}
B_t&=\int_0^t \operatorname{rc}(U_s) \; dB^1_s + \int_0^t
 		\operatorname{sc}(U_s) \; dB^2_s \qquad \text{and}\\
 	\hat{B}_t&=\int_0^t \operatorname{rc}(U_s) \;  (I-2 e_s
 		\esProd{e_s}{\cdot}) \; dB^1_s + \int_0^t \operatorname{sc}(U_s) \; dB^2_s.\nonumber
\end{align}
Since the solutions $(X_t)$ and $(Y_t)$ of \eqref{eq:CouplMcKeanVlasov1} and \eqref{eq:CouplMcKeanVlasov2} are pathwise unique, 
they coincide a.s.\ with the strong solutions of \eqref{eq_mckeanvlasov} w.r.t.\ the Brownian motions $(B_t)$ and $(\hat{B}_t)$ and initial values $x_0$
and $y_0$, respectively. Hence $(X_t,Y_t)$ is indeed a coupling for \eqref{eq_mckeanvlasov}.

\section{Proofs}\label{secProofs}
Let us start with a crucial tool which will be used throughout our proofs: A general construction of the function $f$ 
appearing in the main theorems, characterized by a differential inequality.
\smallskip

We define a concave function $f:[0,\infty)\rightarrow[0,\infty)$ depending on various parameters. Fix constants
$R_1, R_2\in\reals_+$ such that $R_1\leq R_2$, and let functions 
\begin{align*}
h:[0,R_2]\rightarrow [0,\infty), & &j:[0,R_2]\rightarrow [0,\infty) & &\text{and}& &i:[0,R_1]\rightarrow [0,\infty)
\end{align*}
be given. We suppose that $i$ and $j$ are continuous, $j$ is non-decreasing and $h$ is continuously differentiable with $h'\geq 0$. The function
$f$ is given by
\begin{eqnarray*}
f(r)&=&\int_0^{r\wedge R_2} \phi(s)\, g(s)\, ds,
\end{eqnarray*}
where $\phi$ and $g$ are defined as
\begin{align}
\label{phi}	\phi(r)&=\exp(-h(r)) \qquad \text{and} \\ 
\label{g}	g(r)&=1-\frac{\beta}{4}\int_0^{r\wedge R_2}
	j(\,\Phi(s)\,)\,\phi(s)^{-1}\,ds-\frac{\xi}{4} \int_0^{r\wedge R_1} i(s)\,\phi(s)^{-1}\,ds.
\end{align}
Here the function $\Phi$ and the constants $\beta$ and $\xi$ are given by
\begin{equation} \label{Phibetaxi}
	\Phi(r) =\int_0^{r} \phi(s) \, ds, \quad  \beta^{-1}=\int_0^{R_2} j(\,\Phi(s)\,)\,\phi(s)^{-1}\,ds, 
	\quad \xi^{-1}= \int_0^{R_1} i(s)\,\phi(s)^{-1}\,ds.
\end{equation} 
The function $f$ is a generalization of the concave distance function constructed in \cite{Eberle2015}. It is continuously differentiable on $(0,R_2)$ and
constant on $[R_2,\infty)$. The derivative $f'$ on $(0,R_2)$ is given by the product $\phi g$, where $\phi$ and $g$ are positive and non-increasing functions.
Hence $f$ is a concave and non-decreasing function. Notice that $g$ maps the interval $[0,R_2]$ into $[1/2,1]$,  which implies that the following
inequalities hold for any $r\in[0,R_2]$:
\begin{align}\label{eq_f1}
r\, \phi(R_2) &\leq \Phi(r) \leq 2\,f(r)\leq 2 \,\Phi(r) \leq 2 \, r.
\end{align} 
The crucial property of the function $f$ is that it is twice continuously differentiable
on $(0,R_1)\cup(R_1,R_2)$ and that it satisfies on this set the (in)equality
\begin{eqnarray}
\nonumber f''(r)&=&- h'(r) \, f'(r) - \frac{\beta}{4}\,j(\,\Phi(r) \,) -
	\frac{\xi}{4}\, i(r)\, I_{r<R_1}
	\\&\leq& - h'(r)\, f'(r) - \frac{\beta}{4}\, j(\,f(r)\,) -
	\frac{\xi}{4} \,i(r)\, I_{r<R_1}. \label{eq_f2}
\end{eqnarray}
Observe that $f$ is not continuously differentiable at the point $R_2$ and thus we sometimes work
with the left-derivative $f'_{-}$ which exists everywhere. The function $f$ can  formally be extended
to a concave function on $\reals$ by setting $f(r)=-r$ for $r<0$. We can associate with $f$ a signed measure $\mu_f$ on $\reals$, 
which takes the role of a generalized second derivative. For $x<y$ the measure is defined by $\mu_f(\,[x,y)\,)=f_{-}'(y)-f_{-}'(x)$.
On the set $(0,R_1)\cup (R_1,R_2)$ the measure satisfies  
\begin{align*} 
	\mu_f(dx) = f''(x) \, dx,
\end{align*}
since $f$ is twice continuously differentiable. Furthermore,
\begin{align*}
&\mu_f\left((-\infty,0]\cup (R_2,\infty) \right)=0 & \text{and} &  &\mu_f\left(\{R_1,R_2\}\right)\leq 0.
\end{align*}


\subsection{Proofs of results in Section \ref{secGeometricSimple}}\label{secProofsGeometricSimple}

\begin{proof}[Proof of Lemma \ref{lemSufficientCriteriaAddCase1}] 
Let $(x,y)\in\reals^{2d}$ such that $(x,y)\not\in S_2$.
Assume w.l.o.g.\ that $\max\{\eNorm{x},\eNorm{y}\}=\eNorm{x}\geq \mathcal{R}$. Using our assumption, the triangle inequality and the estimate $\Phi(r)\leq r$, we get
\begin{align*} 
	V(x) \geq {4C}{\lambda}^{-1}\, (1+2\eNorm{x})\geq {4C}{\lambda}^{-1} \,(1+\eNorm{x-y})\geq{4C}{\lambda}^{-1} \,(1+\Phi(\eNorm{x-y})).
\end{align*}
\end{proof}   

\begin{proof}[Bounds for Example \ref{exa:exptails}]
A simple computation shows that by \eqref{constdrift} and since $\alpha\le\delta$, 
the Lyapunov function defined in the example satisfies
\begin{eqnarray*}
(\mathcal L V)(x) &\le &\frac{\alpha}{2}\left( h''(|x|)+\alpha h'(|x|)^2+(\frac{d-1}{r}-2\delta )h'(|x|) \right)\, V(x)\\
&\le &\frac{\alpha }{2}\left( h''(|x|)+(\frac{d-1}{r}-\delta )h'(|x|) \right)\, V(x)\ \le\ C-\lambda V(x)
\end{eqnarray*}
both for $|x|\ge 2/\alpha$ and for $|x|<2/\alpha$. Hence \eqref{BR1} holds by \eqref{R1trivialupperbound}. Furthermore, by Lemma \ref{lemSufficientCriteriaAddCase1}, we can choose the set $S_2$ such that
\begin{eqnarray*}
R_2 &=&2\,\sup\{ r\ge 0:\,\exp(\alpha r)<4C\lambda^{-1}(1+2r)\}\\
&=&2\,\sup\{ r\ge 0:\,2r< 2\alpha^{-1}\log (4C\lambda^{-1})+2\alpha^{-1}\log (1+2r)\}\\
&\le &2\,\sup\{ r\ge 0:\,r< \bar R_1+2^{-1}\bar R_1\log (1+r)\}\ \le\ 2(1+\bar R_1)\log (1+\bar R_1).
\end{eqnarray*}
Here we have used that $\log (4C\lambda^{-1})=\log (8e^2)>2$. The last bound holds, since for $1+r=2(1+\bar R_1)\log (1+\bar R_1)$,
\begin{eqnarray*}
\lefteqn{1+\bar R_1+\frac 12\bar R_1\log (1+r)}\\ &=&1+\bar R_1+\frac 12\bar R_1\log (1+\bar R_1) \, +\, \frac 12\bar R_1\log (2\log(1+\bar R_1))\\
&\le &  1+r-\frac 12(1+\bar R_1)\log (1+\bar R_1) \, +\, \frac 12\bar R_1\log (2\log(1+\bar R_1))\\
&\le &  1+r-\frac 12\bar R_1\log (\frac{1+\bar R_1}{2\log(1+\bar R_1}) \ \le\ 1+r.
\end{eqnarray*}
\end{proof}

\begin{proof}[Proof of Theorem \ref{thmAddMainSimple}]
We use the function $f$ defined at the beginning of Section \ref{secProofs} with the following parameters:
The constants $R_1$ and $R_2$ are specified by \eqref{R1additivesimple} and \eqref{R2additivesimple} respectively.
For $r\geq 0$ we set $i(r):=1$, $j(t):=t$ and 
\begin{align}\label{hAdd}
h(r)&:=\frac{1}{2}\int_0^r s\, \kappa(s)\, ds, \quad\text{where $\kappa$ is defined in Assumption \ref{assGlobalCurvBound}. }
\end{align} 
We fix initial values $(x,y)\in\reals^{2d}$ and prove \eqref{eq_contrAddSimple} for Dirac measures
$\delta_x$ and $\delta_y$. This is sufficient, since for general $\mu,\nu\in\mathcal{P}_V$ one can show, arguing similarly to 
\cite[Theorem 4.8]{MR2459454}, that  
for any coupling $\gamma$ of $\mu$ and $\nu$ we have
\begin{eqnarray}
		\wDist{\rho_1}(\mu p_t, \nu p_t) &\leq& \int \wDist{\rho_1}(\delta_x p_t, \delta_y p_t) \, \gamma(dx\,dy).
\end{eqnarray}
 \par
Let $U_t=(X_t,Y_t)$ be a reflection coupling with initial values $(x,y)$, as defined in Section \ref{ReflectionCoupling}. 
We will argue that  
$E\left[e^{c\,t}\rho_1(X_t,Y_t)\right]\leq \rho_1(x,y)$ holds for any $t\geq 0$.
Denote by $T:=\inf\left\{t\geq 0: X_t=Y_t \right\}$ the coupling time.
Set $Z_t=X_t-Y_t$ and $r_t=\eNorm{Z_t}$. The process $(Z_t)$ satisfies the SDE
\begin{eqnarray*}
		dZ_{t}	&=&	(b(X_t)-b(Y_t)) \, dt + 2 \, e_t \, \esProd{e_t}{dB_{t}} \quad\text{for } t<T,\\
	dZ_t 	&=& 0 \quad\text{for } t\geq T, \quad \text{where }e_t={Z_t}/{r_t}.
\end{eqnarray*}
Until the end of the proof, all It{\^o} equations and differential inequalities 
hold almost surely for $t<T$, even though we do not mention it every time.  
An application of It{\^o}'s formula shows that $(r_t)$ satisfies the equation  
\begin{eqnarray*}
	dr_t &=&\esProd{e_t}{b(X_t)-b(Y_t)} \,dt +2 \,\esProd{e_t}{dB_{t}} \quad \text{for $t<T$.}
\end{eqnarray*}
Let $(L_t^x)$ denote the right-continuous local time of the semimartingale $(r_t)$.
Since $f$ is a concave function, we can apply the general It{\^o}-Tanaka formula of Meyer and Wang 
(cf.\ e.g.\ \cite[Thm.\ 22.5]{MR1876169} or \cite[Ch.\ VI]{MR1725357})
to conclude
\begin{eqnarray}
	\nonumber f(r_t)-f(r_0) &=& \int_0^t f'_{-}(r_s)\,\esProd{e_s}{b(X_s)-b(Y_s)}\,ds
	 + 2 \int_0^t f'_{-}(r_s) \,\esProd{e_s}{dB_{s}}
	 \\&&+ \frac{1}{2} \int_{-\infty}^\infty L_t^x\, \mu_f(dx)\qquad\qquad
	 \mbox{for }t<T,\label{eq:fItoAdd}
\end{eqnarray}
where $f_{-}'$ denotes the left-derivative of $f$ and $\mu_f$ is the non-positive measure representing the second derivative of $f$, i.e., $\mu_f(\,[x,y)\,)=f'_{-}(y)-f'_{-}(x)$ for $x\leq y$.
Moreover, the generalized It{\^o} formula implies for every measurable function $v:\reals\rightarrow[0,\infty)$
the equality
\begin{eqnarray}\label{appendix1}
		\int_0^t v(r_s) \, d[r]_s &=& \int_{-\infty}^\infty v(x) \, L_t^x \, dx \quad\text{for any }  t<T.
\end{eqnarray} 
Observe that \eqref{appendix1} implies that the Lebesgue measure of the set
$\{ 0 \leq s \leq T: r_s\in \{R_1,R_2\} \}$, i.e., the time that $(r_s)$ spends at the points $R_1$ and
$R_2$ before coupling, is almost surely zero. Our function
$f$ is twice continuously differentiable on $(0,\infty)\setminus\{R_1,R_2\}$. The measure $\mu_f(dy)$ is non-positive and thus
\eqref{appendix1} implies
\begin{eqnarray*}
\int_{-\infty}^\infty L_t^x \mu_f(dx)&\leq& \int_{-\infty}^\infty I_{{\mathbb R}\setminus\{R_1,R_2\}}(x)f''(x)L_t^x dx = 4 \int_0^t f''(r_s)\, ds, \quad t<T. 
\end{eqnarray*}
We can conclude that a.s.\ the following differential inequalities hold for $t<T$:
\begin{eqnarray*}
df(r_t) &\leq& \left(f'(r_t)\, \esProd{e_t}{b(X_t)-b(Y_t)}+2\, f''(r_t)\right)\,dt + 2\, f'(r_t) \,\esProd{e_t}{dB_{t}}
	\\&\leq&  (-({\beta}/{2}) \,f(r_t)\, I_{r_t<R_2}-({\xi}/{2})\, I_{r_t<R_1})\, dt +  2\, f'(r_t) \,\esProd{e_t}{dB_{t}}.
\end{eqnarray*} 
For the second inequality, we have used that $f$ is constant on $[R_2,\infty)$ and that inequality \eqref{eq_f2}
holds on $(0,R_2)\setminus\{R_1\}$ with $h$ given by \eqref{hAdd}. Moreover, using Assumption \ref{assGlobalCurvBound}, we estimated
\begin{eqnarray*}
\esProd{e_t}{b(X_t)-b(Y_t)}=\esProd{{Z_t}/{r_t}}{b(X_t)-b(Y_t)}&\leq& \kappa(r_t)\, r_t.
\end{eqnarray*}
We now turn to the Lyapunov functions. Assumption \ref{assGeometricDrift} implies that a.s.\
\begin{eqnarray*}
	d\left(\epsilon\,V(X_t)+\epsilon\,V(Y_t)\right)&\leq& 2\,C\,\epsilon dt - \lambda\, \left(\epsilon\,V(X_t)+\epsilon\,V(Y_t)\right) \, dt + dM_t,
\end{eqnarray*}
where $(M_t)$ denotes a local martingale. 
If $r_t\geq R_1$, the definition of $S_1$ implies
\begin{eqnarray*}
	2\,C\,\epsilon - \lambda \, \left(\epsilon\,V(X_t)+\epsilon\,V(Y_t)\right)&\leq& -({\lambda}/{2})\, \left(\epsilon\,V(X_t)+\epsilon\,V(Y_t)\right).
\end{eqnarray*}
If $r_t\geq R_2$, then by Assumption \ref{assAddGrowthCurvBound},
\begin{eqnarray*}
2C\epsilon - \lambda \left(\epsilon V(X_t)+\epsilon V(Y_t)\right)&\leq& -({\alpha}/{2})\,f(r_t)\,  -({\lambda}/{2}) \, \left(\epsilon\,V(X_t)+\epsilon\,V(Y_t)\right),
\end{eqnarray*}
where we have used that by \eqref{eq:epsilon} and \eqref{Phibetaxi}, $\epsilon={\xi}/{(4C)}$ and $\Phi(r)\geq f(r)$. We can conclude that a.s.\ ,
\begin{eqnarray*}
	& &d(\epsilon V(X_t)+\epsilon V(Y_t))
	\\&\leq& \left( ({\xi}/{2}) I_{r_t<R_1} - ({\alpha}/{2}) f(r_t)  I_{r_t\geq R_2} - ({\lambda}/{2}) (\epsilon V(X_t)+\epsilon V(Y_t))\right)dt + dM_t. 
\end{eqnarray*} 
Summarizing the above results, we can conclude that a.s., for $t<T$,
\begin{align}\label{eqp1}
d\rho_1(X_t,Y_t)&=df(r_t)+d\left(\epsilon\,V(X_t)+\epsilon\,V(Y_t)\right)
	\;\leq\; -c \, \rho_1(X_t,Y_t) \, dt + dM_t',
\end{align}
where $(M_t')$ denotes a local martingale and $c={\min\{\alpha,\beta,\lambda\}}/{2}$. 
The product rule for semimartingales then implies a.s.\ for $t<T$:
\begin{eqnarray*} 
		d(e^{c\,t}\rho_1(X_t,Y_t)) &=& c\, e^{c\,t}\rho_1(X_t,Y_t) \,dt + e^{c\,t} \,d\rho_1(X_t,Y_t) \;\leq\;  e^{c\,t} \, dM_t'.
\end{eqnarray*}
We introduce a sequence of stopping times $(T_n)_{n\in\naturals}$ given by
\begin{align}
T_n &:=\inf\{t\geq 0: \eNorm{X_t-Y_t}\leq {1}/{n} \quad\text{or}\quad \max\{\eNorm{X_t},\eNorm{Y_t}\}\geq n\}. 
\end{align}
We have  $T_n\uparrow T$ a.s.\ by non-explosiveness. Therefore we finally obtain:
\begin{align*}
	&\wDist{\addDist_1}\left(\delta_x p_t,\delta_y p_t \right)
	\leq E\left[\addDist_1(X_t,Y_t) I_{t<T}\right] 
	= \lim_{n\rightarrow \infty} E\left[\addDist_1(X_t,Y_t) I_{t<T_n}\right]
\\&\leq
	\operatorname{e}^{-c\,t} \liminf_{n\rightarrow \infty}
	E\left[\operatorname{e}^{c (t\wedge T_n)}\addDist_1(X_{t\wedge T_n},Y_{t\wedge T_n})\right] \leq \operatorname{e}^{-c\,t} 
	\wDist{\addDist_1}(\delta_x,\delta_y) 
\end{align*}
\end{proof} 

\begin{proof}[Bounds for Example \ref{simplifiedconstants}]
	The statement is a special case of Theorem \ref{thmAddMainSimple}. The only thing to verify is the lower bound
	\begin{eqnarray}
		\beta &\geq& \sqrt{{\kappa}/{\pi}} \left(\int_0^{R_2} \exp\left({-\kappa r^2}/{4}\right)\,dr\right)^{-1} .
	\end{eqnarray}	
As $\int_{0}^\infty \exp(-\kappa r^2/4)\, dr=\sqrt{{\pi}/{\kappa}}$, this follows from the definitions of $\Phi$ and $\phi$:
	\begin{eqnarray}
		\beta^{-1}&=&\int_0^{R_2} \phi(r)^{-1}\, \Phi(r) \,dr \;\leq\; \sqrt{{\pi}/{\kappa}} \int_0^{R_2} \exp\left({-\kappa r^2}/{4}\right)\,dr.
	\end{eqnarray} 
\end{proof}
  
\begin{proof} [Proof of Corollary \ref{corAddSimple2}]
It is well-known that, in our setup, the Mar\-kov transition kernels $(p_t)$  admit a unique 
invariant measure $\pi$ satisfying $\pi p_t=\pi$ for any $t\geq 0$ and $\int V(x) \, \pi(dx) \leq {C}/{\lambda}$,
see e.g.\ \cite{hairer2006ergodic}. In \cite[Lemma 2.1]{MR2857021} it is proven that for any probability measures $\nu_1$ and $\nu_2$ we have
\begin{eqnarray*}
\int_{\reals^d} V(x) \, \eNorm{\nu_1-\nu_2}(dx) &=& \inf_{\gamma} \int \left[\, V(x)+V(y)\, \right]\,I_{x\not = y} \, \gamma(dx \, dy),
\end{eqnarray*}
 where the infimum is taken over all couplings $\gamma$ with marginals $\nu_1$ and $\nu_2$ respectively. In our setup, 
 this implies that for any $\mu\in \mathcal{P}_V$ and $t\geq 0$,
 \begin{eqnarray*}
 \int_{\reals^d} V(z) \, \eNorm{\mu p_t-\pi}(dz) &\leq& \epsilon^{-1} \wDist{\rho_1}(\mu p_t, \pi p_t) \; \leq \; \epsilon^{-1}\,e^{-c\, t}\,\wDist{\rho_1}(\mu, \pi).
 \end{eqnarray*}
 This implies the bound on the mixing time, since
 \begin{align*}
 	\wDist{\rho_1}(\delta_x, \pi) &\leq  \int \left[ f(\eNorm{x-y})+\epsilon \, V(x)+\epsilon\, 
		V(y)\, \right] \pi(dy) \leq  R_2+\epsilon \, V(x)+\epsilon\,{C}/{\lambda}.
 \end{align*}
\end{proof}

\begin{proof}[Proof of Corollary \ref{cor:ergodicAvergaesAdditive}]
	Let $x\in\reals^d$. Assumption \ref{assGeometricDrift} implies that $\delta_x p_t\in \mathcal{P}_V$ for any $t\geq 0$ and hence $p_t g(x):=E_x[g(X_t)]$ is well defined and finite for any
	measurable $g$ which is Lipschitz w.r.t.\ $\rho_1$. Fix $(x,y)\in\reals^{2d}$ and $t\geq 0$,
	 and let $(X_t,Y_t)$ be an arbitrary coupling of $\delta_x p_t$ and $\delta_y p_t$. We bound
	the Lipschitz norm of $x\mapsto p_t g(x)$:
	\begin{eqnarray*}
	\eNorm{p_t g(x)-p_t g(y)}&\leq& E_{(x,y)}[\eNorm{g(X_t)-g(Y_t)}]\;\leq\; \|{g}\|_{\operatorname{Lip}(\rho_1)} \, E_{(x,y)}[\rho_1(X_t,Y_t)].
	\end{eqnarray*}
	Since the above inequality holds for any coupling, Theorem \ref{thmAddMainSimple} implies
	\begin{eqnarray*}
		\|{p_t g}\|_{\operatorname{Lip}(\rho_1)} &\leq & \|{g}\|_{\operatorname{Lip}(\rho_1)} \, e^{-c\,t}.
	\end{eqnarray*}
	This estimate implies bounds on the bias of ergodic averages: 
	\begin{align*}
	&\eNorm{E_x\left[\frac{1}{t}\int_0^t g(X_s) \,ds-\int g \, d\pi\right]} \leq \frac{1}{t}\int_0^t \int \eNorm{p_s g(x)-p_s g(y)}
	\pi(dy)\, ds\\
	&\leq \frac{1-e^{-c\,t}}{c\,t}\|g\|_{\operatorname{Lip}(\addDist_1)} \int \rho_1(x,y) \pi(dy)
	\\&\leq \frac{1-e^{-c\,t}}{c\,t}\|g\|_{\operatorname{Lip}(\addDist_1)} \left(R_2+\epsilon V(x)+\epsilon\frac{C}{\lambda}\right),
	\end{align*}
	where we have used that $f$ is bounded by $R_2$.
	\par\smallskip
	We now turn to the variance bound. Integrating \eqref{eq:LyapunovQuadrat} implies
	\begin{eqnarray*}
		E_x[V(X_t)^2]&\leq& {C^\ast}/{\lambda^\ast} + e^{-\lambda^\ast\, t}\, V(x)^2 \quad\text{for any } t\geq 0.
	\end{eqnarray*}
	For reals $a,b,c$, the inequality $(a+b+c)^2\leq 3(a^2+b^2+c^2)$ holds true. Hence
	\begin{eqnarray*}
		\int \int \rho_1(y,z)^2 \, p_t(x,dy) \,p_t(x,dz)&\leq& 3 \left(\,R_2^2+2\,\epsilon^2 \int V(y)^2\, p_t(x,dy)\, \right)
		 \\&\leq& 3 \left(\,R_2^2+2\,\epsilon^2  \left[{C^\ast}/{\lambda^\ast} + e^{-\lambda^\ast\, t}\, V(x)^2 \right] \right).
	\end{eqnarray*}
	Let $A:=3 \left(\,R_2^2+2\,\epsilon^2  \left({C^\ast}/{\lambda^\ast} + e^{-\lambda^\ast\, t}\, V(x)^2 \right) \right)$. For 
	 $t\geq 0$ and $h\geq 0$,
	\begin{align*}
		\operatorname{Var}_x \left[ g(X_t)\right] &= \frac{1}{2}\int \int (\,g(y)-g(z)\,)^2 \, p_t(x,dy)\,p_t(x,dz)
		\leq  \frac{A}{2}\, \|g\|_{\operatorname{Lip}(\addDist_1)}^2,\\
		\operatorname{Var}_x \left[ (p_h g)(X_t)\right] &\leq \frac{A}{2}\,\|{p_h g}\|_{\operatorname{Lip}(\rho_1)}^2 \; \leq \; \frac{A}{2}\, \|g\|_{\operatorname{Lip}(\addDist_1)}^2 \,e^{-2\,c\,h}.
	\end{align*}
We get an estimate on the decay of correlations by Cauchy-Schwarz:
\begin{align*}
	&\operatorname{Cov}_x \left[ g(X_t),g(X_{t+h})\right] = \operatorname{Cov}_x \left[ g(X_t),(p_h g)(X_{t})\right]
	\\&\leq \operatorname{Var}_x \left[ g(X_t)\right]^{1/2} \operatorname{Var}_x \left[ (p_h g)(X_t)\right]^{1/2}
	\;\leq\; \frac{A}{2}\, \|g\|_{\operatorname{Lip}(\addDist_1)}^2\, e^{-c\,h}.
\end{align*}
Finally, we obtain the variance bound
\begin{align*}
	\operatorname{Var}_x\left[\frac{1}{t}\int_0^t g(X_s) ds\right]&=
	\frac{2}{t^2}\int_0^t \int_r^t \operatorname{Cov}_x\left[\, g(X_s)\,,g(X_r)\,\right] \, ds\,dr
	\\&\le\frac{A}{t^2}  \, \|g\|_{\operatorname{Lip}(\addDist_1)}^2 \, \int_0^t \int_r^t  e^{-c\,(s-r)}\, ds\,dr
	\;=\;\frac{A}{c\, t}  \|g\|_{\operatorname{Lip}(\addDist_1)}^2.
\end{align*}
\end{proof}

\subsection{Proofs of results in Section
\ref{secMultiplicative}}\label{secProofsMultiplicative}

\begin{proof}[Proof of Theorem \ref{thmMultMainSimply}]
We use the function $f$ defined at the beginning of Section \ref{secProofs} with the following parameters:
The constants $R_1$ and $R_2$ are specified by \eqref{R1multsimple}, we fix 
$\epsilon\in(0,\infty)$ satisfying \eqref{eq:alphaMult}, set $i(r):=\Phi(r)$ and $j(r):=r$ and define  
\begin{eqnarray}\label{hMult}
	h(r)&:=&\frac{1}{2}\int_0^r s\, \kappa(s)\, ds + 2\,\multFunc(\epsilon)\,r, 
\end{eqnarray}
where $\Phi$, $\kappa$ and $\multFunc$ are given by \eqref{eq:phimult}, Assumption \ref{assGlobalCurvBound} and \eqref{def:funck}.
\smallskip

Fix initial values $(x,y)\in\reals^{2d}$. It is enough to prove \eqref{eq_contrMultSimple} for Dirac measures
$\delta_x$ and $\delta_y$, see the proof of Theorem \ref{thmAddMainSimple} for details.
Let $U_t=(X_t,Y_t)$ be a reflection coupling with initial values $(x,y)$, as defined in Section \ref{ReflectionCoupling}. 
We will argue that  
$E\left[e^{c\,t}\rho_2(X_t,Y_t)\right]\leq \rho_2(x,y)$ holds for any $t\geq 0$.
Denote by $T=\inf\left\{t\geq 0: X_t=Y_t \right\}$ the coupling time.
Set $Z_t=X_t-Y_t$ and $r_t=\eNorm{Z_t}$. 
The proof of Theorem \ref{thmAddMainSimple} shows that $f(r_t)$ satisfies a.s.
\begin{align}\label{dfMult}
df(r_t) &\leq \left(f'(r_t)\, \esProd{e_t}{b(X_t)-b(Y_t)}+2 f''(r_t)\right)\,dt + 2\, f'(r_t) \,\esProd{e_t}{dB_{t}}
\end{align}
for $t<T$, where $e_t={Z_t}/{r_t}$. As in the proof of Theorem \ref{thmAddMainSimple}, the Lebesgue measure of the set 
$\{ 0 \leq s \leq T: r_s\in \{R_1,R_2\} \}$, i.e.\,  the time that $(r_t)$ spends at the points $R_1$ and
$R_2$ before coupling, is almost surely zero. This justifies to write $f'$ and $f''$ in the above inequality.
Observe that Assumption  \ref{assGlobalCurvBound} implies the upper bound
$$\esProd{e_t}{b(X_s)-b(Y_s)}=\esProd{{Z_t}/{r_t}}{b(X_t)-b(Y_t)}\leq \kappa(r_t)\, r_t.$$
The function $f$ is constant on $[R_2,\infty)$, and $f(r)\leq \Phi(r)$. Moreover, on $(0,R_2)\setminus\{R_1\}$ the function $f$ satisfies
inequality \eqref{eq_f2}. By \eqref{dfMult}, \eqref{eq_f2} and \eqref{hMult}, we can conclude that a.s.\ for $t<T$,
\begin{align}\label{eq:fmult2}
df(r_t) 
	&\leq  (-4\,\multFunc(\epsilon)\,f'(r_t)-{\beta}/{2} f(r_t)\, I_{r_t<R_2}-{\xi}/{2} f(r_t)\, I_{r_t<R_1})\, dt
	\\&+  2\, f'(r_t) \,\esProd{e_t}{dB_{t}}.\nonumber
\end{align}
\par\smallskip
We now turn to the Lyapunov functions 
and set $G(x,y):=1+\epsilon\,V(x)+\epsilon\,V(y)$. 
By definition of the coupling in Section \ref{ReflectionCoupling}, we have a.s.\ for $t<T$:
\begin{eqnarray}
	\label{eq:VMult} \lefteqn{dG(X_t,Y_t) \ =\ \left(\epsilon\,\generator V(X_t)+\epsilon\,\generator V(Y_t)\right) \, dt }
	\\&&+ \epsilon \esProd{\nabla V(X_t)+\nabla V(Y_t)}{dB_t} - 2\, \epsilon \, \esProd{e_t}{\nabla V(Y_t)} \esProd{e_t}{dB_t}.\nonumber
\end{eqnarray}
Assumption \ref{assGeometricDrift} implies 
$\generator V(X_t)+\generator V(Y_t) \leq 2\,C - \lambda\, (V(X_t)+V(Y_t))$. 
Notice that by  \eqref{eq:alphaMult}, \eqref{eq:phimult} and \eqref{Phibetaxi} with $i(r)=\Phi(r)$, 
\begin{equation}
2C\epsilon \leq \left(2\int_0^{R_1} \Phi(r)\,\phi(r)^{-1}\, dr\right)^{-1}=\xi/2.
\end{equation}
Recall  
that   $c=\min\left\{\beta,\lambda,4C\epsilon\lambda \right\}/2$.
 Using the definitions \eqref{eq:S1mult} and \eqref{eq:S2mult} of the sets $S_1$ and $S_2$ respectively, 
we can conclude that a.s.\ for $t<T$: 
\begin{align}\nonumber
	d(\epsilon V(X_t)+\epsilon V(Y_t)&)\leq \left({\xi}/{2}I_{r_t<R_1} - c \,G(X_t,Y_t) I_{r_t\geq R_2}\right)\,dt 
	\\&\quad + \epsilon \esProd{\nabla V(X_t)+\nabla V(Y_t)}{dB_t} - 2\epsilon \esProd{e_t}{\nabla V(Y_t)} \esProd{e_t}{dB_t}.\label{lypMult}
\end{align} 
By \eqref{eq:fItoAdd} and \eqref{eq:VMult}, the covariation of $f(r_t)$ and $\epsilon \, V(X_t)+\epsilon\,V(Y_t)$ is, almost surely for $t<T$, given by:
\begin{eqnarray*}
		d[\,f(r)\,, \epsilon\,V(X)+\epsilon\,V(Y)\,]_{t}&=&2\,f'(r_t)\,\epsilon\,\esProd{\nabla V(X_t)-\nabla V(Y_t)}{e_t} \,dt.
\end{eqnarray*} 
%
%
Using Cauchy-Schwarz and \eqref{def:funck}, we can derive the following bound for any $x,y\in\reals^d$ with $x\not=y$:
\begin{align*} 
	\epsilon\esProd{\nabla V(x)-\nabla V(y)}{\frac{x-y}{\eNorm{x-y}}} &\leq
	(1+\epsilon\,V(x)+\epsilon\,V(y)) \frac{\epsilon \eNorm{\nabla V(x)}+\epsilon \eNorm{\nabla V(y)}}{(1+\epsilon\,V(x)+\epsilon\,V(y))}
	\\&\leq 2\,\multFunc(\epsilon)\,G(x,y). 
\end{align*}
Hence, almost surely for $t<T$:
\begin{eqnarray}\label{quadVarMult}
	d[\,f(r)\,, \epsilon\,V(X)+\epsilon\,V(Y)\,]_{t}&\leq& 4\,\multFunc(\epsilon)\,f'(r_t)\,G(X_t,Y_t)\,dt.
\end{eqnarray}
The product rule for semimartingales implies almost surely for $t<T$:
\begin{align*} 
d\left(f(r_t)G(X_t,Y_t)\right) &= G(X_t,Y_t) \, df(r_t) + f(r_t) \, dG(X_t,Y_t) + [f(r),G(X,Y)]_t.
\end{align*}
By \eqref{eq:fmult2}, we have
\begin{eqnarray}\nonumber
	  G(X_t,Y_t)df(r_t) &\leq&  
	\left(-{\beta}/{2}\,\rho_2(X_t,Y_t)\,I_{r_t<R_2}-{\xi}/{2}\,\rho_2(X_t,Y_t)\,I_{r_t<R_1}\right) dt
	\\\label{l1}&&- 4\,\multFunc(\epsilon)\,f'(r_t)G(X_t,Y_t)\, dt + dM_t^1,
\end{eqnarray} 
where $(M_t^1)$ is a local martingale. 
Moreover, \eqref{lypMult} implies
\begin{equation}\label{l2}
	 f(r_t) \, dG(X_t,Y_t) \ \leq\  [{\xi}/{2}\, f(r_t) \, I_{r_t<R_1} - c\,  \rho_2(X_t,Y_t)\, I_{r_t\geq R_2}] \, dt + dM_t^2,
\end{equation} 
where $(M_t^2)$ is again a local martingale. Observe that 
$G\geq 1$.
Combining \eqref{quadVarMult}, \eqref{l1} and \eqref{l2}
we can conclude a.s.\ for $t<T$:
\begin{eqnarray*}
	d\rho_2(X_t,Y_t) &\leq& - {c\,\rho_2(X_t,Y_t)} + dM_t,\\  
	d\left(e^{c\,t} \rho_2(X_t,Y_t)\right) &=& c \,e^{c\,t}\, \rho_2(X_t,Y_t) \, dt+ e^{c\,t} \, d\rho_2(X_t,Y_t) \;\leq\; e^{c\,t}\,dM_t,
\end{eqnarray*}
where $(M_t)$ is a local martingale.
We can finish the proof of \eqref{eq_contrMultSimple} using a stopping argument, see the end of the proof of Theorem \ref{thmAddMainSimple}
for details.
\end{proof} 

\begin{proof}[Proof of Corollary \ref{corMulti22}]
	Analogously to the proof of Corollary \ref{cor:ergodicAvergaesAdditive}, we can conclude that
	$p_t g(x)$ is finite for any function $g$ which is Lipschitz w.r.t.\ $\rho_2$, any $x\in\reals^d$ and $t\geq 0$. Moreover,
	\begin{eqnarray*}
		\|{p_t g}\|_{\operatorname{Lip}(\rho_2)} &\leq&  \|{g}\|_{\operatorname{Lip}(\rho_2)} \, e^{-c\,t} \qquad\text{holds for any } t\geq 0.
	\end{eqnarray*}
	In particular, for any $x,y\in\reals^d$ we can conclude that
	\begin{eqnarray*}
			\eNorm{p_tg(x)-p_tg(y)} &\leq& \|{g}\|_{\operatorname{Lip}(\rho_2)} \, e^{-c\,t} \rho_2(x,y) 
		\\&\leq& \|{g}\|_{\operatorname{Lip}(\rho_2)} \, e^{-c\,t} \eNorm{x-y}\left(1+\epsilon \,V(x)+\epsilon\, V(y)\right),
	\end{eqnarray*}
	where we used $f(r)\leq r$. If the map $x\mapsto p_tg(x)$ is differentiable at $x\in\reals^d$, we can deduce the gradient bound \eqref{gradientbound}.
\end{proof}

\subsection{Proofs of results in Section
\ref{secMckeanvlasov}}\label{proofMcKeanVlasov}

\begin{proof}[Proof of Theorem \ref{thmMcKeanConvexAtInfinity}] 
In contrast to the proofs above, we do not use the function $f$
defined in the beginning of Section \ref{secProofs},
but the one 
constructed in \cite{Eberle2015}, i.e., we set 
\begin{align*} 
	f(r)&=\int_0^{r} \phi(s)\, g(s\wedge R_2)\, ds,
\end{align*} 
where $\phi$ and $g$ are defined as
\begin{align*}
	\phi(r)&=\exp\left(- \frac{1}{2} \int_0^r u\, \kappa^+(u) \, du \right) & &\text{and} & 
	g(r)&=1-\frac{c}{2}\int_0^{r}
	\Phi(s)\,\phi(s)^{-1}\,ds.
\end{align*}
The function $\Phi$ and the constant $c$ are given by
\begin{align*}
	\Phi(r)& =\int_0^{r} \phi(s) \, ds & &\text{and} &  c^{-1}&=\int_0^{R_2} \Phi(s)\,\phi(s)^{-1}\,ds.
\end{align*}
The constants $R_1$ and $R_2$ are defined in \eqref{eq:defR1} and \eqref{eq:defR2} respectively. Notice
that by definition, $\kappa^+(r)=0$ for any $r\geq R_1$ and thus $f$ is linear on the interval $[R_2,\infty)$.
The function $f$ is twice continuously differentiable on $(0,R_2)$, and
\begin{align}\label{eq:McKeanproofeq1}
2 \, f''(r) = - \, r\, \kappa^+(r)\,f'(r) -\, c \,\Phi(r)\leq  -\, r\, \kappa^+(r)\,f'(r) - c \,f(r).
\end{align} 
We now prove \eqref{mckeancontr1} and fix initial probability measures $\mu_0$ and $\nu_0$ as well as a small constant $\delta>0$. We thus have $X_0\sim\mu_0$ and $Y_0\sim \nu_0$. We first define the initial coupling. We assume, as is usual in contraction results, that $\wDist{\rho_0}(\mu_0,\nu_0)=E(\rho_0(X_0,Y_0))$.
The coupling $U_t:=(X_t,Y_t)$, defined in Section \ref{sec:couplingMcKean}, yields the upper bound 
$$\wDist{\rho_0}(\mu_t^{x},\mu_t^{y})\ \leq\  E[\rho_0(X_t,Y_t)]\ =\ E[f(\eNorm{X_t-Y_t})].$$ 
Let $\gamma:=c-\eNorm{\tau}\,K$. Set $Z_t:=X_t-Y_t$ and
$r_t:=\eNorm{Z_t}$. We will argue that there is a constant $C>0$, independent of $\delta$, such that
\begin{align}\label{eq:proofMcKeanVlasov1}
e^{\gamma\, t} E[f(r_t)]\leq f(r_0) + e^{\gamma\, t}\, C\, \delta \qquad\text{holds true for any $t\geq 0$}.
\end{align}
From this inequality one can then conclude, that for any $t\geq 0$ we have
$$\wDist{\rho_0}(\mu_t^,\nu_t)\leq e^{-\gamma\, t}\, \wDist{\rho_0}(\mu_0,\nu_0) + C\,\delta,$$
which finishes the proof of \eqref{mckeancontr1} since $\delta>0$ can be chosen arbitrarily small. 
Moreover \eqref{mckeancontr2} directly follows from \eqref{mckeancontr1} and the inequality
\begin{align*}
r\, \phi(R_1) &\leq \Phi(r) \leq 2\,f(r)\leq 2 \,\Phi(r) \leq 2 \, r.
\end{align*} 
We now show \eqref{eq:proofMcKeanVlasov1}. 
By definition of the coupling in Section \ref{sec:couplingMcKean},
$$ 	dZ_t\ =\ \left(b^{\mu_0}(t,X_t)-b^{\nu0}(t,Y_t)\right)\, dt + 2\, 
		\operatorname{rc}(U_t) \, e_t \, dW_t,
$$
where $W_t=\int_0^t \esProd{e_s}{dB^1_s}$ is a one dimensional Brownian motion. 
	Notice that whenever $r_t<\delta/2$, we have $\operatorname{rc}(U_t)=0$ by definition. Using an approximation
	argument, cf.\ \cite[Proof of Lemma 3]{zimmer16} or arguing similarly to \cite[Lemma
6.2]{Eberle2015}, one can show that $r_t$ satisfies almost surely the equation
	\begin{align}\label{eqMcKeanThmr1}
		dr_t&=\esProd{\tilde{e}_t}{b^{\mu_0}(t,X_t)-b^{\nu_0}(t,Y_t)}\, dt + 2\, \operatorname{rc}(U_t)\, dW_t,
	\end{align}
	where $\tilde{e}_t:=Z_t/r_t$ for $r_t\not=0$, $\tilde{e}_t:=(b^{\mu_0}(t,X_t)-b^{\nu_0}(t,Y_t))/\eNorm{b^{\mu_0}(t,X_t)-b^{\nu_0}(t,Y_t)}$ if $r_t=0$ and $\eNorm{b^{\mu_0}(t,X_t)-b^{\nu_0}(t,Y_t)}>0$,
	and $\tilde{e}_t$ is an arbitrary unit vector otherwise.
	Similarly as in the proof in Section \ref{secProofsGeometricSimple}, we now apply the {It\^o-Tanaka} for semimartingales to conclude that almost surely,
\begin{eqnarray*}
	f(r_t)-f(r_0) &=& \int_0^t f'_{-}(r_s) \,\esProd{\tilde{e}_s}{b^{\mu_0}(s,X_s)-b^{\nu_0}(s,Y_s)}\, ds 
	\\& &+\, 2\, \int_0^t \operatorname{rc}(U_s)\, f'_{-}(r_s) \,dW_s
	+\frac{1}{2} \int_{-\infty}^{\infty} L_t^x\, \mu_f(dx),
\end{eqnarray*}
where $L_t^x$ is the right-continuous local time of $(r_t)$ and $\mu_f$ is the non-positive measure representing the second derivative of $f$. By
\eqref{appendix1},
the Lebesgue measure  
of the set $\{ 0 \leq s \leq t : r_s \in \{R_1,R_2\} \}$ is almost surely zero.
Since $f$ is twice continuously differentiable, except possibly at $R_1$ and $R_2$, we can replace $f'_{-}$ by $f'$ 
in the equation above. Moreover, since $f$ is concave, the measure of the points $R_1$ and $R_2$ w.r.t.\ $\mu_f$ is non-positive. Hence by \eqref{appendix1},
\begin{eqnarray}\int_{-\infty }^\infty L_t^x\, \mu_f(dx)& \le & \int_0^tf''(r_s)\, d[r]_s\ =\ 4\int_0^t\operatorname{rc}(U_s)^2 \, f''(r_s)\, ds\quad\mbox{a.s., and thus}\nonumber\\
\label{eqMcKeanThm1f}
		f(r_t)&  =& f(r_0)\, +\, M_t\, +\, \int_0^t H_s\, ds,\qquad\ \mbox{where}\\
		M_t &=& 2\int_0^t \operatorname{rc}(U_s) f'(r_s) dW_s,\qquad
		\mbox{ and}\label{eq:Mtrc}\\
		H_s &\le & 
		f'(r_s)\, \esProd{\tilde{e}_s}{b^{\mu_0}(s,X_s)-b^{\nu_0}(t,Y_s)} \, + \, 2\,\operatorname{rc}(U_s)^2 \, f''(r_s) .\label{eq:Atrc}
	\end{eqnarray}	
%
	We can bound the inner product using the definitions of $b^{\mu_0}$, $b^{\nu_0}$ and $\kappa$, 
	as well as the Lipschitz bounds on $b$ and $\vartheta$:
	{\small
	\begin{align}\label{est1}
&\esProd{\tilde{e}_t}{b^{\mu0}(t,X_t)-b^{\nu0}(t,Y_t)}   
	\\
	&=\esProd{\tilde{e}_t}{b(X_t)-b(Y_t)}+\tau \esProd{\tilde{e}_t}{\int \vartheta(X_t,z) \mu_t(dz)-\int \vartheta(Y_t,z) \nu_t(dz)}\nonumber \\\nonumber 
	 &\leq  
	  I_{r_t\geq \delta}  r_t \kappa(r_t) + I_{r_t< \delta} \gNorm{b}_{\operatorname{Lip}}  \delta 
	  +\eNorm{\tau} L (r_t + \wDist{}^1(\mu_t,\nu_t)).
	\end{align}}
	Notice that 
	$\wDist{}^1(\mu_t,\nu_t) \leq E[r_t]$. Remembering that by \eqref{mckean1K}, $K=\frac{4L}{\phi(R_1)}$ and combining \eqref{est1} with
	the inequality 
	$r\leq 2\,f(r)/\phi(R_1)$, we obtain
	\begin{eqnarray}\nonumber
		\lefteqn{\esProd{\tilde{e}_t}{b^{\mu_0}(t,X_t)-b^{\nu_0}(t,Y_t)}}  \\
		&\leq& I_{r_t\geq \delta} \, r_t\, \kappa(r_t) + I_{r_t< \delta}\, \gNorm{b}_{\operatorname{Lip}} \, \delta 
		+ \eNorm{\tau}\,
		{K}/{2}\, (f(r_t) + E[f(r_t)]).\label{mm1}
	\end{eqnarray}
	The product rule for semimartingales shows that
	$$
	 	d (e^{\gamma\,t}\, f(r_t) ) \ =\ e^{\gamma t} \, dM_t 
	 	\, + \, e^{\gamma t}\left(\gamma \, f(r_t)\, +\, H_t \right)\,  dt.
	$$
	Using that $\gamma=c-\eNorm{\tau}K$ and the bound $f'\leq 1$, we can conclude that
	 \begin{eqnarray}\nonumber
	 d (e^{\gamma\,t}\, f(r_t) ) &\leq& e^{\gamma t} \, dM_t + e^{\gamma t}\, \eNorm{\tau}\,{K}/{2}\,\left(E[f(r_t)]-f(r_t)\right)  dt
	 	\\ \label{eq70a}&&+  e^{\gamma t}\, I_{r_t<\delta} \left( c \,f(r_t)+\gNorm{b}_{\operatorname{Lip}} \, \delta \right) dt 
	 	\\\nonumber &&+ e^{\gamma t} \, I_{r_t\geq\delta} \left( c\, f(r_t)\,+ r_t\, \kappa(r_t)\,f'(r_t)\,+ 2 f''(r_t) \right) dt.
	 \end{eqnarray}
	 Here we used that $f''\leq 0$ and $\operatorname{rc}(U_t)=1$ whenever $r_t\geq \delta$. We now argue that for any $r\in (0,\infty)\setminus\{R_2\}$ we have
	\begin{align}\label{proofMcKeaneq2}
		c\, f(r) +  r\, \kappa(r) \, f'(r)+2 \, f''(r) \leq 0.
	\end{align}
	For $r\in(0,R_2)$ this inequality follows directly from the definition of $f$, see \eqref{eq:McKeanproofeq1}. 
	For $r>R_2$ we have $f''(r)=0$, but $\kappa(r)$ is sufficiently negative instead: First notice that for $r\geq R_1$, $\phi(r)$
	is constant and hence
$\Phi(r)=\Phi(R_1)+\phi(R_1)(r-R_1)$.
Analogously to \cite[Theorem 2.2.]{Eberle2015} we get
\begin{align*}
	\int_{R_1}^{R_2} \Phi(s)\phi(s)^{-1} ds &= \int_{R_1}^{R_2}
	(\Phi(R_1)+\phi(R_1)(s-R_1))\phi(R_1)^{-1} ds\\
	&= \Phi(R_1)\phi(R_1)^{-1}(R_2-R_1)+(R_2-R_1)^2/2
	\\&\geq (R_2-R_1)(\Phi(R_1)+\phi(R_1)(R_2-R_1))\phi(R_1)^{-1}/2
	\\&=(R_2-R_1)\Phi(R_2)\phi(R_1)^{-1}/2.
\end{align*} 
For $r\geq R_2$ we have $f'(r)=\phi(R_1)/2$, and 
thus we get
\begin{align*}
	f'(r) \, r\, \kappa(r) &\leq - 2\frac{\phi(R_1) 
	}{R_2-R_1}\frac{r}{R_2}\leq - 2\frac{\phi(R_1)  
	}{R_2-R_1}\frac{\Phi(r)}{\Phi(R_2)}
	\leq -c\, \Phi(r)\leq -c \, f(r),
\end{align*} 
where we used the definition of $R_2$ in \eqref{eq:defR2} and the fact that $c^{-1}=\int_0^{R_2}\Phi(s)/\phi(s)\,ds$. 
Hence \eqref{proofMcKeaneq2} holds 
for any $r\in (0,\infty)\setminus\{R_2\}$. By \eqref{eq70a}, we conclude that
\begin{align*}
	E[e^{\gamma\,t}f(r_t)-f(r_0)] 
	&\leq \delta \left(\gNorm{b}_{\operatorname{Lip}} + c\right)\,\int_0^t e^{\gamma s} \, ds,
\end{align*}
where we used that $f(r)\leq r$.
\end{proof} 
  
We now prepare the proof of Theorem \ref{thmMcKeanDrift} by providing a priori bounds. Notice that 
Assumption \ref{assMcKeanVlasovLipschitz} implies that there are
constants $A,B>0$ s.t.
\begin{equation}\label{eq:mckeanlinearupperbound}
	\eNorm{\vartheta (x,y)}\leq A + B \,
	(\,\eNorm{x}+\eNorm{y}\, ) \qquad\text{for any } x,y\in\reals^d.
\end{equation} 

\begin{lem}[A priori bounds]\label{lemMcKeanApriori}
	Let $V(x)=1+\eNorm{x}^2$. Suppose that Assumptions
	\ref{assMcKeanVlasovLipschitz} and \ref{assMcKeanDrift} hold true. Then there is a constant $C\in (0,\infty )$ such that for
	any $\tau\in\reals$ with $\eNorm{\tau}\leq {\lambda}/{(8B)}$, $x\in \reals^d$ and $t\geq 0$, a solution $(X_t)$ of \eqref{eq_mckeanvlasov}
	with $X_0=x$ satisfies	
	{\small
\begin{eqnarray*}
			 dV(X_t) 
&\leq& \left[\left(C - \lambda \,V(X_t)\right)  +\left(2\eNorm{\tau}B\eNorm{X_t}E[\eNorm{X_t}]-\frac{\lambda}{4}\eNorm{X_t}^2\right) \right]\, dt  + 2 \esProd{X_t}{dB_t}.	  	 
\end{eqnarray*}} 
	 In particular,\qquad
	 $
	 	E\left[V(X_t)\right]\ \leq\ {C}/{\lambda} + e^{-\lambda\, t}\, E[V(X_0)].
	 $
\end{lem}

\begin{proof}[Proof of Lemma \ref{lemMcKeanApriori}]
Let $M_t:=\int_0^t \esProd{X_s}{dB_s}$. 
By  {It\^o}'s formula,
\begin{equation}
\label{eq:VXt}\frac{1}{2}\,dV(X_t) = \esProd{X_t}{b(X_t)} \, dt + \tau \langle X_t, \int \vartheta(X_t,y) \,\mu^x_t(dy)\rangle\, dt + \frac d2 \, dt + dM_t.
\end{equation}
Using Assumption \ref{assMcKeanDrift}, inequality \eqref{eq:mckeanlinearupperbound} and $\eNorm{\tau}\leq \lambda/(8B)$, we conclude
\begin{align*}
\frac{1}{2}dV(X_t) &\leq [ C_1 - \lambda \eNorm{X_t}^2 + \eNorm{\tau} (A \eNorm{X_t} + B (\eNorm{X_t}^2 + \eNorm{X_t} E_x[\eNorm{X_t}]))]dt
+ dM_t\\
&\leq \left[C_2 - \frac{5}{8}\lambda \eNorm{X_t}^2 + \eNorm{\tau} B \eNorm{X_t} E_x[\eNorm{X_t}]\right] dt
+ dM_t,
\end{align*}
with $C_1:=\sup_{\eNorm{x}\leq D} \eNorm{\esProd{x}{b(x)}+\lambda \eNorm{x}^2+d}$ and a constant $C_2>C_1$ s.t. 
\begin{eqnarray*}
	- {\lambda \, r^2}/{4}  + \eNorm{\tau} \,A \,r &\leq& C_2-C_1 \quad \text{for any } r\in\reals_+.
\end{eqnarray*}
It follows that we can find a constant $C>0$ such that
\begin{align}\label{eq:mckeanvlasovlyapunovcond}
dV(X_t) 
&\leq \left[C - \lambda V(X_t)  + 2 \eNorm{\tau}\,B\,\eNorm{X_t}\, E[\eNorm{X_t}]\,-\frac{\lambda}{4}\,\eNorm{X_t}^2\right] dt
+ 2 dM_t.
\end{align}
Applying the product rule for semimartingales we get
\begin{align*}
	d(e^{\lambda\,t}\,V(X_t)) &\leq e^{\lambda\,t}\, \left[C +2\,\eNorm{\tau}\,B\,\eNorm{X_t}\, E[\eNorm{X_t}]\,-\frac{\lambda}{4}\,\eNorm{X_t}^2\right]\, dt + 2\, e^{\lambda\,t}\,dM_t.
\end{align*}
We introduce the stopping times $T_n:=\inf\{t\geq 0: \eNorm{X_t}\geq n\}$ and remark that almost surely, $T_n\uparrow\infty$, 
since the solution $(X_t)$ is non-explosive. 
Using Fatou's Lemma and monotone convergence, we can conclude that
\begin{align*}
	E_x[e^{\lambda\, t} \, V(X_t)]&\leq \liminf_{n\rightarrow\infty} E_x[e^{\lambda\, (t\wedge T_n)} \, V(X_{t\wedge T_n})] \\
					&\leq V(X_0) + \int_0^{t} e^{\lambda\,s}\, \left[C +2\, \eNorm{\tau}\,B\, E[\eNorm{X_s}]^2\,-\frac{\lambda}{4}\,E[\eNorm{X_s}^2]\right]\, ds.
\end{align*}
This concludes the proof, since by assumption, $\eNorm{\tau}\leq {\lambda}/{(8B)}$. 
\end{proof} 
 
\begin{proof}[Proof of Theorem \ref{thmMcKeanDrift}]
We use the Lyapunov function $V(x)=1+\eNorm{x}^2$. 
Assumption \ref{assMcKeanDrift} provides a rate $\lambda$ and 
Lemma \ref{lemMcKeanApriori} a constant $C$. 
We follow Section \ref{secMultiplicative} defining 
\begin{eqnarray*} 
S_1 & :=& \left\{(x,y)\in \reals^d\times\reals^d:
\lypFunc(x)+\lypFunc(y)\leq {2C}/{\lambda}\right\},\\
S_2& :=& \left\{(x,y)\in \reals^d\times\reals^d:
\lypFunc(x)+\lypFunc(y)\leq {8C}/{\lambda}\right\},\\
R_i &:=& \sup\left\{ \eNorm{x-y}:(x,y)\in S_i\right\}, \quad i=1,2.
\end{eqnarray*}
We define $f$ as in the beginning of Section \ref{secProofs} w.r.t.\ the following parameters:
\begin{equation*}
h(s):=\frac{1}{2}\int_0^s r \kappa(r) dr + 2s, j(s)=s, i(s):=\Phi(s), c:=\frac{1}{4}\min\{{\beta},{\lambda}\}, \epsilon:=\frac{\xi}{4C},
\end{equation*}
with $\kappa$ given by Assumption \ref{assGlobalCurvBound}.
We assume $\eNorm{\tau}<{\lambda}/{(8B)}$ so that Lemma \ref{lemMcKeanApriori} applies.
\smallskip

Fix initial probability measures $\mu_0$ and $\nu_0$, as well as a small constant $\delta>0$. As before, let $X_0\sim\mu_0$ and $Y_\sim \nu_0$ such that $\wDist{\rho_2}(\mu_0,\nu_0)=E(\rho_2(X_0,Y_0)$.
The coupling $U_t:=(X_t,Y_t)$ defined in Section \ref{sec:couplingMcKean} yields the upper bound
$$\wDist{\rho_2}(\mu_t^{x},\mu_t^{y})\leq E[\rho_2(X_t,Y_t)].$$ 
Set $Z_t:=X_t-Y_t$ and $r_t:=\eNorm{Z_t}$. Equations \eqref{eqMcKeanThmr1}, \eqref{eqMcKeanThm1f}, \eqref{eq:Mtrc} and
\eqref{eq:Atrc} are still valid 
in our setup. 
By \eqref{est1} we can conclude that
$$
	H_t  \leq  \left(I_{r_t\geq \delta}  f'(r_t) \kappa(r_t)  r_t+ I_{r_t< \delta} \gNorm{b}_{\operatorname{Lip}}  \delta \right)
		+2  \operatorname{rc}(U_t)^2 f''(r_t) 
		+ \eNorm{\tau} L  (r_t + E[r_t]) .
$$
By definition, $f$ is constant on $[R_2,\infty)$, and, for $r\in (0,R_2)\setminus\{R_1\}$,
$$2\,f''(r)\leq - f'(r)\,[\kappa (r)\,r \,+ 4 ] - ({\beta}/{2})\, f(r) - ({\xi}/{2})\, f(r) \,I_{r<R_1}.$$
Using that $f$ is concave with $f(r)\leq r$ and $\operatorname{rc}(U_t)=1$ for $r_t\geq \delta$, we obtain
\begin{eqnarray}\nonumber
1_{r_t\le R_2}d f(r_t) &\leq & \left[- \frac{\beta}{2}\, f(r_t) I_{r_t<R_2} - \frac{\xi}{2}\, f(r_t) \,I_{r_t<R_1}- 4 \operatorname{rc}(U_t)^2\, f'(r_t) \right]\, dt 
	\\&&+ \left( \gNorm{b}_{\operatorname{Lip}}+\frac{\beta}{2}+\frac{\xi}{2}\right) \, \delta \, dt
	+  f'(r_t) \eNorm{\tau}\, L \, (r_t + E[r_t]) \, dt  \\&&\nonumber+ 2\, \operatorname{rc}(U_t)\, f'(r_t) \,dW_t.\label{eq:f}
\end{eqnarray}
Moreover, for $r_t>R_2$, $f(r_t)$ is a constant and thus $df(r_t)=0$. Next, we observe that Lemma \ref{lemMcKeanApriori} implies 
\begin{align*} 
	  dV(X_t) &\leq \left[C - \lambda \,V(X_t)\right]\, dt + 2\eNorm{\tau}\,B\,V(X_t)\, E[V(X_t)]\,dt+ 2 \esProd{X_t}{dB_t},\\
	  dV(Y_t) &\leq \left[C - \lambda \,V(Y_t)\right]\, dt + 2\eNorm{\tau}\,B\,V(Y_t)\, E[V(Y_t)]\,dt+ 2 \esProd{Y_t}{d\hat{B}_t},
\end{align*}
where $(B_t)$ and $(\hat{B}_t)$ are the Brownian motions defined in \eqref{BMcoupling}. Let 
$$G(x,y):=1+\epsilon\,V(x)+\epsilon\,V(y).$$
The set $S_1$ is chosen such that
$
	2\,C\,\epsilon-\lambda\,\epsilon\, V(X_t)-\lambda\,\epsilon\, V(Y_t)\leq 0$ whenever $r_t\geq R_1$.
For $r_t\geq R_2$ we have
\begin{align*}
 2\,C\,\epsilon -\lambda \,\epsilon \,V(X_t)+\lambda \,\epsilon \,V(Y_t)
 &\leq -2\,C\,\epsilon  - ({\lambda}/{2})\, \epsilon\, V(X_t)-({\lambda}/{2})\, \epsilon\, V(Y_t)
 \\&\leq - \min\{{\beta}/{2},{\lambda}/{2}\}\,G(X_t,Y_t),
\end{align*}  
since $\epsilon={\xi}/{(4C)}$ and $\xi\geq \beta$. 
We conclude that
\begin{align}\nonumber
	dG(X_t,Y_t) &\leq  I_{r_t<R_1} 2\, C \,\epsilon - I_{r_t\geq R_2} \min\{{\beta}/{2},{\lambda}/{2}\}\,G(X_t,Y_t)\, dt+\epsilon\,2 \esProd{X_t}{dB_t} 
	\\&+2 \epsilon \eNorm{\tau}B\left[ V(X_t)\, E[V(X_t)] + V(Y_t)\, E[V(Y_t)] \right]dt
	+  \epsilon\, 2 \esProd{Y_t}{d\hat{B}_t}.\label{eq:G}
\end{align}
Note that $|\nabla V(x)|=2|x|\le V(x)$. Therefore, and
by \eqref{eqMcKeanThm1f} and \eqref{eq:Mtrc}, we obtain similarly to \eqref{quadVarMult}:
\begin{eqnarray}\nonumber
d[f(r),G(X,Y)]_t &=&   2\, \operatorname{rc}(U_t)^2\, f'(r_t)\,
\epsilon\langle\nabla V(X_t)-\nabla V(Y_t),\, e_t\rangle\, dt\\
&\le &2\, \operatorname{rc}(U_t)^2\, f'(r_t)\,
	G(X_t,Y_t) \, dt.\label{eq:covariation}
\end{eqnarray}
Using the product rule together with \eqref{eq:f},\eqref{eq:G} and
\eqref{eq:covariation}, we see that
\begin{eqnarray}\nonumber
	\lefteqn{d\rho_2(X_t,Y_t)\ =\ d\left(f(r_t)G(X_t,Y_t)\right) }\\
	\nonumber &=& G(X_t,Y_t)df(r_t)+f(r_t)dG(X_t,Y_t) + d[f(r),G(X,Y)]_t
	\\\nonumber&\leq &- \min\{{\beta}/{2},{\lambda}/{2}\}\, f(r_t) \,G(X_t,Y_t) \, dt 
	+  \eNorm{\tau}\, L \, G(X_t,Y_t)\,(r_t + E[r_t]) \, dt
		\\\label{product} &&+2\,\epsilon\,\eNorm{\tau}\,B\,f(r_t) \left[ V(X_t)\, E[V(X_t)] + V(Y_t)\, E[V(Y_t)] \right]\,dt
		\\\nonumber&&+ G(X_t,Y_t)\left( \gNorm{b}_{\operatorname{Lip}}+({\beta}+{\xi})/{2}\right) \, \delta \, dt+ d\tilde M_t,
\end{eqnarray} 
where $(\tilde M_t)$ denotes a local martingale. 
We further bound the perturbation terms originating from the non-linearity.
For $r<R_2$, inequality \eqref{eq_f1} holds true
and thus there is a constant $K_0\in(0,\infty)$ s.t.\ 
$$
\eNorm{x-y}\ \leq\ K_0 \, f(\eNorm{x-y}),\qquad \mbox{if}\, |x-y|\le R_2
$$
and for any $x,y\in\reals^d$ we have
$$
\eNorm{x-y}\ \leq\ K_0 \, f(\eNorm{x-y})\,(1+\epsilon\, V(x)+\epsilon \,V(y))
\ =\ K_0\, \rho_2(x,y).
$$
Hence, we can bound 
\begin{align*}
	\eNorm{\tau}\, L \, G(X_t,Y_t)\,(r_t + E[r_t])  \leq \eNorm{\tau}\, L \, K_0  \,(\rho_2(X_t,Y_t)+G(X_t,Y_t) E[\rho_2(X_t,Y_t)]).
\end{align*}
Moreover,
$$
\epsilon\,V(X_t)\, E[V(X_t)]+\epsilon\,V(Y_t)\, E[V(Y_t)]\ \leq\  \epsilon^{-1}\, E[G(X_t,Y_t)] \, G(X_t,Y_t).
$$
Recall that $2c=\min\{{\beta}/{2},{\lambda}/{2}\}$. Hence by the bounds above,
$$	d(e^{c\,t}\rho_2(X_t,Y_t)) \ =\ c \,\rho_2(X_t,Y_t) \,e^{c\,t}\, dt + e^{c\,t} d\rho_2(X_t,Y_t)\ \le\ e^{ct}J_t \, dt\, +\, e^{c\,t} d\tilde M_t, $$
where
\begin{align*}
J_t& = 	 - c  \rho_2(X_t,Y_t) \,  
							+\,\eNorm{\tau}\, L \, K_0  \,(\rho_2(X_t,Y_t)+G(X_t,Y_t) E[\rho_2(X_t,Y_t)]) \\
						&	+\, 2\,\eNorm{\tau}\,B\,
							\epsilon^{-1}\, E[G(X_t,Y_t)] \, \rho_2(X_t,Y_t)   \, 
+G(X_t,Y_t)\left( \gNorm{b}_{\operatorname{Lip}}+({\beta}+{\xi})/{2}\right)  \delta  .				
\end{align*}
Optional stopping and Fatou's lemma now shows that
$$
	E[e^{c\,t}\rho_2(X_t,Y_t)]\ \leq\   \rho_2(X_0,Y_0) + \int_0^{t}e^{cs}\, E[J_s] \,ds. 
$$
Using the a priori bounds from Lemma \ref{lemMcKeanApriori}, we see that there is a constant $C_1\in(0,\infty)$, not depending on $\delta$, such that
\begin{align*}
\left( \gNorm{b}_{\operatorname{Lip}}+({\beta}+{\xi})/{2}\right) \, \int_0^{t}  e^{c\,s}\, E[G(X_s,Y_s)] \, ds\ \leq\ C_1.
\end{align*}
Since $G\geq 1$, we can conclude that
\begin{eqnarray*}
	\lefteqn{\eNorm{\tau}\, L \, K_0  \, \int_0^t \left(E[\rho_2(X_s,Y_s)]+
 E[G(X_s,Y_s)] E[\rho_2(X_s,Y_s)]\right) \,e^{c\,s}\,ds} \\&&+ 2\,\eNorm{\tau}\,B\,
							\epsilon^{-1}\,\int_0^t E[G(X_s,Y_s)] \, E[\rho_2(X_s,Y_s)]   \, e^{c\,s}\,ds
							\\&\leq &  \eNorm{\tau} C_2 \int_0^t E[G(X_s,Y_s)] \, E[\rho_2(X_s,Y_s)]   \,e^{c\,s}\, ds,
\end{eqnarray*} 
where $C_2:=2\left(L\, K_0+{B}/{\epsilon}\right)$. Moreover, the a priori estimates imply 
\begin{align*}
	&\int_0^{t} e^{c\,s}\,  \,E[G(X_s,Y_s)]\, E[\rho_2(X_s,Y_s)]\, ds
	\\&\leq C_3 \,  \int_0^{t} e^{c\,s}\, E[\rho_2(X_s,Y_s)]\, ds + C_4(x,y) \, \int_0^{t} e^{(c-\lambda)\,s}\, E[\rho_2(X_s,Y_s)]\, ds,
\end{align*} 
where $C_3:=1+\epsilon\,2\,{C}/{\lambda}$ and $C_4:=\epsilon\, \mu_0(V)+\epsilon\,\nu_0(V)$. If $\tau$ is sufficiently small,
i.e., if $\eNorm{\tau} C_2 (C_3+C_4)\leq c$, we can conclude that for any $\delta>0$,
\begin{align*}
	\wDist{\rho_2}(\mu_t,\nu_t)\leq E[\rho_2(X_t,Y_t)] &\leq e^{-c\,t}\,\wDist{\rho_2}(\mu_0,\nu_0) + C_1\,\delta.
\end{align*}
However, observe that $C_4$ depends on the initial probability measures, i.e. we get a local contraction in the sense that for a given $R>0$, 
we can find a constant $\tau_0\in(0,\infty)$, such that
\eqref{mckeancontr3} holds for all $\eNorm{\tau}\leq \tau_0$ and initial probability measures $\mu_0,\nu_0$, with $\mu_0(V),\nu_0(V)\leq R$.
Inequality \eqref{mckeancontr4} follows readily from \eqref{mckeancontr3} and the definition of $K_0$. 
\smallskip

In order to obtain a related statement which is valid for any initial condition, see \eqref{mckeancontr5}, we assume
$\eNorm{\tau} C_2 C_3 <c$. Similarly as above, we
obtain
\begin{align*}
	 &E[\rho_2(X_t,Y_t)] 
	 \\&\leq e^{-c\,t}\wDist{\rho_2}(\mu_0,\nu_0) + C_1\,\delta +e^{-c\,t} \eNorm{\tau}\, C_2\,C_4 \, \int_0^{t} e^{(c-\lambda)\,s}\, E[\rho_2(X_s,Y_s)]\, ds
\end{align*}
Using once again the apriori estimates and the bound $f\leq R_2$, we see that
\begin{align*}
	\int_0^{t} e^{(c-\lambda)\,s} E[\rho_2(X_s,Y_s)] ds \leq R_2 (1+2\,\epsilon\, C/\lambda + \epsilon \mu_0(V) + \epsilon \nu_0(V)) \int_0^{t} e^{(c-\lambda)\,s} ds.
\end{align*} 
Since $\lambda>c$, there is a constant $K_1\in(0,\infty)$, neither depending on the initial values $(x,y)$ nor on $\delta$, such that
\begin{align*}
	 E[\rho_2(X_t,Y_t)] &\leq e^{-c\,t}\,\wDist{\rho_2}(\mu_0,\nu_0) + C_1\,\delta +e^{-c\,t} \, K_1 \,(\epsilon \mu_0(V)+\epsilon \nu_0(V))^2.
\end{align*}
Since $\delta>0$ is arbitrary, we have shown \eqref{mckeancontr5}. 
\end{proof}


\subsection{Proofs of results in Section
\ref{secSubGeometricSimple}}\label{secProofsSubGeometricSimple}

Before proving Theorem \ref{thmMainSubgeometric}, we include a 
proof of Lemma \ref{subgeometricKnownResult} in Section
\ref{dis:subgeometric} that is based on \cite[Section 4]{hairerlecturenotes}.

\begin{proof}[Proof of Lemma \ref{subgeometricKnownResult}]
The function $H$ is $C^2$ with strictly positive first derivative, and thus the
	inverse function $H^{-1}:[0,\infty]\rightarrow [l,\infty]$ is also strictly increasing and $C^2$. 
	We define a function $G:[l,\infty)\times [0,\infty)\rightarrow [0,\infty)$
	by
	\begin{align*}
	G(x,t):=H^{-1}(\,H(x)\,+c\,t\,).
	\end{align*}
	Observe  that for any fixed $t\geq 0$ the map 
	 $x\mapsto G(x,t)$ is a concave $\contFunctions^2$ function on $(l,\infty)$, which can be seen by the following computation:
	$$
		\partial_x^2 \,G\	= \partial_x \left(\frac{\eta\circ G}{\eta }\right) \ =\
		\frac{(\eta\eta')\circ G }{\eta^2}-\frac{(\eta\circ G)\, \eta'}{\eta^2}\ \leq\ 0.
	$$
Since $x\mapsto G(x,t)$ is concave, It{\^o}'s formula shows that almost surely,\
	\begin{align*}
		dG(Z_t,t) \ \leq\ \partial_t G(Z_t,t) \, dt + \partial_x G(Z_t,t) \, dA_t + dW_t,
	\end{align*}
	where $(W_t)$ denotes a local martingale. 
	Observe that $
	 	\partial_t \, G=c \,\eta\circ G >0\ $ and $\ \partial_x\,
	 	G= \frac{\eta\circ G}{\eta} > 0$.
	Using our Assumption \eqref{lemsubass}, we can conclude that a.s.\ 
	\begin{align*}
		dG(Z_t,t) \leq  dW_t \qquad\text{ for } t<T.
	\end{align*}
	Let $(T_n)_{n\in\naturals}$  be a localizing sequence for $(W_t)$ with $T_n\uparrow\infty$.
	We see
	 \begin{align*}
	 	&E[\,G(\,Z_{t\wedge T}\,,t\wedge T\,)\,]
	 	=E[\,\liminf_{n\rightarrow
	 	\infty} \,G(\,Z_{t\wedge T\wedge T_n}\,,t\wedge T\wedge T_n\,)\,]
	 	\\&\leq \liminf_{n\rightarrow
	 	\infty}
	 	E\left[\,G(\,Z_{t\wedge T\wedge T_n},\,t\wedge T\wedge T_n\,)\,\right]
	 	\leq E[\,G(\,Z_0,\,0\,)\,]
	 	= E[\,Z_0\,]. 
	 \end{align*}
	 Since $H$ is non-negative and $H^{-1}$ is increasing, we get
	 $$E[\,H^{-1}(\,c \,(t\wedge T)\,)\,]\leq E[\,G(\,Z_{t\wedge T},\,t\wedge T\,)\,]\leq
	 E[\,Z_0\,]<\infty.$$ 
	 Since the inequality holds for any $t\geq 0$ and $H^{-1}(t)\rightarrow\infty$
	as $t\rightarrow\infty$, the time $T$ is a.s.\ finite, and we can finish the proof using Fatou's lemma:
	 \begin{align*}
	 	E[\,H^{-1}(\,c \,T\,)\,]\leq E[\,G(\,Z_{T},\,T\,\,)]&\leq \liminf_{t\rightarrow\infty}E[\,G(\,Z_{t\wedge T},\,t\wedge T\,)\,]\leq
	 	E[\,Z_0\,].
	 \end{align*}
\end{proof}
	
\begin{proof}[Proof of Theorem \ref{thmMainSubgeometric}]
We use the function $f$ defined in the beginning of Section \ref{secProofs} with the following parameters:
$i\equiv 1$ constant,  $j=\eta$, and $
	h(r):=\frac{1}{2}\int_0^r s\, \kappa(s)\, ds$, where $\kappa$ is defined in Assumption \ref{assGlobalCurvBound}. 
	
We now prove \eqref{eqSubMain}.
Let $U_t=(X_t,Y_t)$ be a reflection coupling with initial values $(x,y)$, as defined in Section \ref{ReflectionCoupling}. 
Denote by $T:=\inf\left\{t\geq 0: X_t=Y_t \right\}$ the coupling time.
We will argue that the stochastic process $(\rho_1(X_t,Y_t))$ satisfies the conditions of Lemma \ref{subgeometricKnownResult}, except that the map $t\mapsto \rho_1(X_t,Y_t)$ is not continuous at $t=T$. Nevertheless, this obstacle can be overcome 
by a stopping argument. Set $Z_t=X_t-Y_t$ and $r_t=\eNorm{Z_t}$. 
Following the lines of the proof of Theorem \ref{thmAddMainSimple}, one can show that  a.s.\ for $t<T$, $f(r_t)$ satisfies
\begin{align*}
	df(r_t) &\ \leq\ \left[f'(r_t)\, \esProd{e_t}{b(X_t)-b(Y_t)}+2 f''(r_t)\right]\,dt + 2\, f'(r_t) \,\esProd{e_t}{dB_{t}}
	\\&\ \leq\  [-{\beta}/{2} \, \eta(\, f(r_t)\,)\, I_{r_t<R_2}-{\xi}/{2}\, I_{r_t<R_1}]\, dt +  2\, f'(r_t) \,\esProd{e_t}{dB_{t}}.
\end{align*}
We turn to the Lyapunov functions. Assumption \ref{assSubDrift} implies that a.s.,
\begin{align*}
	d(\epsilon\,V(X_t)+\epsilon\,V(Y_t))\ \leq\ \left(2C\epsilon - (\epsilon\,\eta(\,V(X_t)\,)+\epsilon\,\eta(\,V(Y_t)\,)) \right)\, dt + dM_t, 
\end{align*}
where $(M_t)$ is a local martingale. Observe that by definition of $\gamma$ in Theorem \ref{thmMainSubgeometric}, and by concavity of
$\eta$, we have
\begin{align*}
	\epsilon (\,\eta(\,V(X_t)\,)+\eta(\,V(Y_t)\,)) \geq 	\epsilon \,\eta(\,V(X_t)+V(Y_t)\,) \geq \gamma \,\eta(\,\epsilon\, V(X_t)+\epsilon\,V(Y_t)\,).
\end{align*}
 If $r_t\geq R_1$, we know by definition of $S_1$ that
$$2C\epsilon - \, (\epsilon\,\eta(\,V(X_t)\,)+\epsilon\,\eta(\,V(Y_t)\,))\leq -\,{\gamma}/{2} \, \eta(\,\epsilon\, V(X_t)+\epsilon\,V(Y_t)\,).$$
If $r_t\geq R_2$, then by Assumption \ref{assAddGrowthCurvBoundSubgeometric} and since $\eta$ is increasing,
$$2C\epsilon - (\epsilon\, \eta (V(X_t))+\epsilon\,\eta (V(Y_t)))\leq -{\alpha}/{2}\,\eta(f(r_t))  - {\gamma}/{2}\,  \eta(\epsilon V(X_t)+\epsilon V(Y_t)),$$
where we have used that $\epsilon=\min (1,{\xi}/{(4C)})$ and $\Phi\geq f$. Thus a.s., 
\begin{eqnarray*} 
	d(\epsilon V(X_t)+\epsilon V(Y_t))&\leq & \left({\xi}/{2} I_{r_t<R_1} - {\alpha}/{2} \,\eta(\,f(r_t)\,)\, I_{r_t\geq R_2}\right)dt
	\\ &&- {\gamma}/{2} \, \eta(\,\epsilon\, V(X_t)+\epsilon\,V(Y_t)\,) \, dt + dM_t. 
\end{eqnarray*}
Summarizing the above results, we can conclude that almost surely, the following differential inequality holds for $t<T$:
\begin{align}\nonumber
	d\rho_1(X_t,Y_t)&=df(r_t)+d(\epsilon\,V(X_t)+\epsilon\,V(Y_t))
	\\&\leq  -{\min\{\alpha,\beta\}}/{2}\,\,\eta(\,f(r_t)\,)- {\gamma}/{2} \, \eta(\epsilon\, V(X_t)+\epsilon\,V(Y_t)) \, dt + dM_t'\nonumber
	\\&\leq -{\min\{\alpha,\beta,\gamma\}}/{2} \, \eta(\, \rho_1(X_t,Y_t) \,)\, dt + dM_t',
\end{align} 
where $(M_t')$ denotes a local martingale and ${\min\{\alpha,\beta,\gamma\}}/{2}=c$.
\par
Now let $T_n:=\inf\{t\geq 0: \eNorm{X_t-Y_t}\leq \frac{1}{n}\}$. By non-explosiveness we have $T_n\uparrow T$.
We have shown that the semimartingale $R_t:=\rho_1(X_{t\wedge T_n},Y_{t\wedge T_n})$ satisfies the assumptions
of Lemma \ref{subgeometricKnownResult} for the stopping time $T_n$. Thus 
\begin{align*} 
	E[\,H^{-1}(\,c \,T\,)\,]&\leq \liminf_{n\rightarrow \infty} E[\,H^{-1}(\,c \,T_n\,)\,]
	\\&\leq \liminf_{n\rightarrow \infty} E\left[\,H^{-1}(\,H(\,R_{T_n}\,) +
		c \,T_n\, )\right] \leq E[\,R_0\,] = \rho_1(x,y).
\end{align*}
Inequality \eqref{eqSubMain} now follows from an application of the Markov inequality, and by the fact that $H^{-1}$ is strictly increasing:
\begin{align*}
	P[ T > t ] = P[\, H^{-1}(c\,T) > H^{-1}(c\,t)\,] \leq \frac{E[H^{-1}(c\,T)]}{H^{-1}(c\,t)}\leq \frac{\rho_1(x,y)}{H^{-1}(c\,t)}.
\end{align*}
\end{proof}

\begin{proof}[Proof of Corollary \ref{corsub4}]
The proof is similar to the one of \cite[Theorem 4.1]{hairerlecturenotes}. 
Consider the probability measure 
$\pi_R(\cdot):=\pi(\cdot\cap A_R)/\pi(A_R)$ where $A_R:=\{x\in\reals^d: V(x)\leq R\}$ for some constant $R\in(0,\infty)$ to be determined below. Since $\pi p_t=\pi$,
\begin{align*}
	\lefteqn{\gNorm{ p_t(x,\cdot ) - \pi}_{TV}
	\ \leq\ \int \gNorm{ p_t(x,\cdot ) - p_t(y,\cdot)}_{TV} \,\pi_R(dy)+\gNorm{ \pi_R \,p_t - \pi \,p_t}_{TV} }
	\\&\leq \frac{\int \rho_1(x,y) \, \pi_R(dy)}{H^{-1}(c\,t)} + \pi(A_R^c) 
	\ \leq \ \frac{R_2+\epsilon\, V(x) + \epsilon \int V(y)\, \pi_R(dy)}
{H^{-1}(c\, t)} + \pi(A_R^c),  
\end{align*}
where we have used that $f\leq R_2$. Similarly to \cite[Lemma 4.1]{MR3178490}, one can see that Assumption \ref{assSubDrift}
implies that the invariant measure $\pi$ satisfies \-
$\int \eta(\,V(y)\,) \, \pi(dy)\leq C$.
Hence, the Markov inequality implies \- $\pi(A_R^c) \leq {\lypConstant}/{\concFunc(R)}$. 
Since
$x\mapsto \concFunc(x)/x$ is non-increasing we have
$$V(x)\leq
{\concFunc(\,V(x)\,)\,R}\,/\,{\concFunc(\,R\,)}$$ for any $x\in\reals^d$ such that $V(x)\leq R$.
This yields the upper bound
	$$\int_{V\leq R} V \, d\pi \ \leq \ {C\,R}/{\concFunc(R)}.$$ 
We can conclude that 
$$
	\gNorm{ p_t(x,\cdot ) - \pi}_{TV} \ \leq \ \frac{R_2+\epsilon\,V(x)}{H^{-1}(c\,t)}
	+ \frac{\epsilon\, C\, R} {\pi(A_R)\,\concFunc(R)\,H^{-1}(c t)} +
	\frac{\lypConstant}{\concFunc(R)}.
$$
We now choose $R$. Set $b:={\eta^{-1}(\,2\,C\,)}/{l}$ 
and define $R:=b\,H^{-1}(c\,t)$. Since \newline $\eta(b\,H^{-1}(0))=\eta(\,b\, l\,)=2\,C$, 
we also have a lower bound for $\pi(A_R)$:
$$\pi(A_R)\ =\ 1-\pi(A_R^c)\ \geq
\	1-{\lypConstant}/{\concFunc(R)}\ \geq\ {1}/{2}.$$
Combining the bounds, we obtain the assertion
\begin{align*}
	\gNorm{ p_t(x,\cdot ) - \pi}_{TV} &\ \leq\ \frac{R_2+\epsilon \, V(x)}{H^{-1}(\,c \, t\,)}
	+ \frac{C\,(\,2\,\epsilon\, b +1\,)} {\concFunc(\,b\, H^{-1}(c\, t))\,}.
\end{align*}
\end{proof}


\def\cprime{$'$}


\begin{thebibliography}{63}
\providecommand{\natexlab}[1]{#1}
\providecommand{\url}[1]{\texttt{#1}}
\expandafter\ifx\csname urlstyle\endcsname\relax
  \providecommand{\doi}[1]{doi: #1}\else
  \providecommand{\doi}{doi: \begingroup \urlstyle{rm}\Url}\fi

\bibitem[Bakry et~al.(2008{\natexlab{a}})Bakry, Barthe, Cattiaux, and
  Guillin]{MR2386063}
D.~Bakry, F.~Barthe, P.~Cattiaux, and A.~Guillin.
\newblock A simple proof of the {P}oincar\'e inequality for a large class of
  probability measures including the log-concave case.
\newblock \emph{Electron. Commun. Probab.}, 13:\penalty0 60--66,
  2008{\natexlab{a}}.

\bibitem[Bakry et~al.(2008{\natexlab{b}})Bakry, Cattiaux, and
  Guillin]{MR2381160}
D.~Bakry, P.~Cattiaux, and A.~Guillin.
\newblock Rate of convergence for ergodic continuous {M}arkov processes:
  {L}yapunov versus {P}oincar\'e.
\newblock \emph{J. Funct. Anal.}, 254\penalty0 (3):\penalty0 727--759,
  2008{\natexlab{b}}.

\bibitem[Bakry et~al.(2014)Bakry, Gentil, and Ledoux]{MR3155209}
D.~Bakry, I.~Gentil, and M.~Ledoux.
\newblock \emph{Analysis and geometry of {M}arkov diffusion operators}, volume
  348 of \emph{Grundlehren der Mathematischen Wissenschaften [Fundamental
  Principles of Mathematical Sciences]}.
\newblock Springer, Cham, 2014.

\bibitem[Bobkov(2003)]{bobkov}
S.~G. Bobkov.
\newblock Spectral gap and concentration for some spherically symmetric
  probability measures.
\newblock In \emph{Geometric aspects of functional analysis}, volume 1807 of
  \emph{Lecture Notes in Math.}, pages 37--43. Springer, Berlin, 2003.

\bibitem[Bolley et~al.(2012)Bolley, Gentil, and Guillin]{MR2964689}
F.~Bolley, I.~Gentil, and A.~Guillin.
\newblock Convergence to equilibrium in {W}asserstein distance for
  {F}okker-{P}lanck equations.
\newblock \emph{J. Funct. Anal.}, 263\penalty0 (8):\penalty0 2430--2457, 2012.

\bibitem[Bolley et~al.(2013)Bolley, Gentil, and Guillin]{MR3035983}
F.~Bolley, I.~Gentil, and A.~Guillin.
\newblock Uniform convergence to equilibrium for granular media.
\newblock \emph{Arch. Ration. Mech. Anal.}, 208\penalty0 (2):\penalty0
  429--445, 2013.

\bibitem[Butkovsky(2014{\natexlab{a}})]{MR3178490}
O.~Butkovsky.
\newblock Subgeometric rates of convergence of {M}arkov processes in the
  {W}asserstein metric.
\newblock \emph{Ann. Appl. Probab.}, 24\penalty0 (2):\penalty0 526--552,
  2014{\natexlab{a}}.

\bibitem[Butkovsky(2014{\natexlab{b}})]{MR3403022}
O.~A. Butkovsky.
\newblock On ergodic properties of nonlinear {M}arkov chains and stochastic
  {M}c{K}ean-{V}lasov equations.
\newblock \emph{Theory Probab. Appl.}, 58\penalty0 (4):\penalty0 661--674,
  2014{\natexlab{b}}.

\bibitem[Carrillo et~al.(2003)Carrillo, McCann, and Villani]{MR2053570}
J.~A. Carrillo, R.~J. McCann, and C.~Villani.
\newblock Kinetic equilibration rates for granular media and related equations:
  entropy dissipation and mass transportation estimates.
\newblock \emph{Rev. Mat. Iberoamericana}, 19\penalty0 (3):\penalty0 971--1018,
  2003.

\bibitem[Carrillo et~al.(2006)Carrillo, McCann, and Villani]{MR2209130}
J.~A. Carrillo, R.~J. McCann, and C.~Villani.
\newblock Contractions in the 2-{W}asserstein length space and thermalization
  of granular media.
\newblock \emph{Arch. Ration. Mech. Anal.}, 179\penalty0 (2):\penalty0
  217--263, 2006.

\bibitem[Cattiaux et~al.(2008)Cattiaux, Guillin, and Malrieu]{MR2357669}
P.~Cattiaux, A.~Guillin, and F.~Malrieu.
\newblock Probabilistic approach for granular media equations in the
  non-uniformly convex case.
\newblock \emph{Probab. Theory Related Fields}, 140\penalty0 (1-2):\penalty0
  19--40, 2008.

\bibitem[Cattiaux et~al.(2009)Cattiaux, Guillin, Wang, and Wu]{MR2498560}
P.~Cattiaux, A.~Guillin, F.-Y. Wang, and L.~Wu.
\newblock Lyapunov conditions for super {P}oincar\'e inequalities.
\newblock \emph{J. Funct. Anal.}, 256\penalty0 (6):\penalty0 1821--1841, 2009.

\bibitem[Cattiaux et~al.(2010)Cattiaux, Guillin, and Wu]{MR2653230}
P.~Cattiaux, A.~Guillin, and L.-M. Wu.
\newblock A note on {T}alagrand's transportation inequality and logarithmic
  {S}obolev inequality.
\newblock \emph{Probab. Theory Related Fields}, 148\penalty0 (1-2):\penalty0
  285--304, 2010.

\bibitem[Chen and Li(1989)]{MR972776}
M.~F. Chen and S.~F. Li.
\newblock Coupling methods for multidimensional diffusion processes.
\newblock \emph{Ann. Probab.}, 17\penalty0 (1):\penalty0 151--177, 1989.

\bibitem[Dalalyan(2017)]{Dalalyan}
A.~S. Dalalyan.
\newblock Theoretical guarantees for approximate sampling from smooth and
  log-concave densities.
\newblock \emph{J. R. Stat. Soc. Ser. B. Stat. Methodol.}, 79\penalty0
  (3):\penalty0 651--676, 2017.

\bibitem[Douc et~al.(2004)Douc, Fort, Moulines, and Soulier]{MR2071426}
R.~Douc, G.~Fort, E.~Moulines, and P.~Soulier.
\newblock Practical drift conditions for subgeometric rates of convergence.
\newblock \emph{Ann. Appl. Probab.}, 14\penalty0 (3):\penalty0 1353--1377,
  2004.

\bibitem[Douc et~al.(2009)Douc, Fort, and Guillin]{MR2499863}
R.~Douc, G.~Fort, and A.~Guillin.
\newblock Subgeometric rates of convergence of {$f$}-ergodic strong {M}arkov
  processes.
\newblock \emph{Stochastic Process. Appl.}, 119\penalty0 (3):\penalty0
  897--923, 2009.

\bibitem[{Durmus} and {Moulines}(2016)]{DurmusMoulinesB}
A.~{Durmus} and E.~{Moulines}.
\newblock {Sampling from strongly log-concave distributions with the Unadjusted
  Langevin Algorithm}.
\newblock \emph{ArXiv e-prints}, May 2016.

\bibitem[Durmus and Moulines(2017)]{DurmusMoulines}
A.~Durmus and E.~Moulines.
\newblock Nonasymptotic convergence analysis for the unadjusted {L}angevin
  algorithm.
\newblock \emph{Ann. Appl. Probab.}, 27\penalty0 (3):\penalty0 1551--1587,
  2017.

\bibitem[Eberle(2011)]{MR2843007}
A.~Eberle.
\newblock Reflection coupling and {W}asserstein contractivity without
  convexity.
\newblock \emph{C. R. Math. Acad. Sci. Paris}, 349\penalty0 (19-20):\penalty0
  1101--1104, 2011.

\bibitem[Eberle(2015)]{Eberle2015}
A.~Eberle.
\newblock Reflection couplings and contraction rates for diffusions.
\newblock \emph{Probability Theory and Related Fields}, pages 1--36. Published
  online (DOI) 10.1007/s00440--015--0673--1, 2015.

\bibitem[Ferr{\'e} et~al.(2013)Ferr{\'e}, Herv{\'e}, and Ledoux]{MR3076780}
D.~Ferr{\'e}, L.~Herv{\'e}, and J.~Ledoux.
\newblock Regular perturbation of {$V$}-geometrically ergodic {M}arkov chains.
\newblock \emph{J. Appl. Probab.}, 50\penalty0 (1):\penalty0 184--194, 2013.

\bibitem[Fort and Roberts(2005)]{MR2134115}
G.~Fort and G.~O. Roberts.
\newblock Subgeometric ergodicity of strong {M}arkov processes.
\newblock \emph{Ann. Appl. Probab.}, 15\penalty0 (2):\penalty0 1565--1589,
  2005.

\bibitem[Hairer(2002)]{MR1939651}
M.~Hairer.
\newblock Exponential mixing properties of stochastic {PDE}s through asymptotic
  coupling.
\newblock \emph{Probab. Theory Related Fields}, 124\penalty0 (3):\penalty0
  345--380, 2002.

\bibitem[Hairer(2006)]{hairer2006ergodic}
M.~Hairer.
\newblock Ergodic properties of markov processes.
\newblock \emph{Lecture notes, Univ. Warwick.
  http://www.hairer.org/notes/Markov.pdf}, 2006.

\bibitem[Hairer(2016)]{hairerlecturenotes}
M.~Hairer.
\newblock Convergence of {M}arkov processes.
\newblock \emph{Lecture notes, Univ. Warwick.
  http://www.hairer.org/notes/Convergence.pdf}, 2016.

\bibitem[Hairer and Mattingly(2008)]{MR2478676}
M.~Hairer and J.~C. Mattingly.
\newblock Spectral gaps in {W}asserstein distances and the 2{D} stochastic
  {N}avier-{S}tokes equations.
\newblock \emph{Ann. Probab.}, 36\penalty0 (6):\penalty0 2050--2091, 2008.

\bibitem[Hairer and Mattingly(2011)]{MR2857021}
M.~Hairer and J.~C. Mattingly.
\newblock Yet another look at {H}arris' ergodic theorem for {M}arkov chains.
\newblock In \emph{Seminar on {S}tochastic {A}nalysis, {R}andom {F}ields and
  {A}pplications {VI}}, volume~63 of \emph{Progr. Probab.}, pages 109--117.
  Birkh\"auser/Springer Basel AG, Basel, 2011.

\bibitem[Hairer et~al.(2011)Hairer, Mattingly, and Scheutzow]{MR2773030}
M.~Hairer, J.~C. Mattingly, and M.~Scheutzow.
\newblock Asymptotic coupling and a general form of {H}arris' theorem with
  applications to stochastic delay equations.
\newblock \emph{Probab. Theory Related Fields}, 149\penalty0 (1-2):\penalty0
  223--259, 2011.

\bibitem[Hairer et~al.(2014)Hairer, Stuart, and Vollmer]{MR3262508}
M.~Hairer, A.~M. Stuart, and S.~J. Vollmer.
\newblock Spectral gaps for a {M}etropolis-{H}astings algorithm in infinite
  dimensions.
\newblock \emph{Ann. Appl. Probab.}, 24\penalty0 (6):\penalty0 2455--2490,
  2014.

\bibitem[Harris(1956)]{MR0084889}
T.~E. Harris.
\newblock The existence of stationary measures for certain {M}arkov processes.
\newblock In \emph{Proceedings of the {T}hird {B}erkeley {S}ymposium on
  {M}athematical {S}tatistics and {P}robability, 1954--1955, vol. {II}}, pages
  113--124. University of California Press, Berkeley and Los Angeles, 1956.

\bibitem[Has{\cprime}minski{\u\i}(1960)]{MR0133871}
R.~Z. Has{\cprime}minski{\u\i}.
\newblock Ergodic properties of recurrent diffusion processes and stabilization
  of the solution of the {C}auchy problem for parabolic equations.
\newblock \emph{Teor. Verojatnost. i Primenen.}, 5:\penalty0 196--214, 1960.

\bibitem[Hou et~al.(2005)Hou, Liu, and Zhang]{MR2157514}
Z.~Hou, Y.~Liu, and H.~Zhang.
\newblock Subgeometric rates of convergence for a class of continuous-time
  {M}arkov process.
\newblock \emph{J. Appl. Probab.}, 42\penalty0 (3):\penalty0 698--712, 2005.

\bibitem[Joulin(2009)]{MR2543873}
A.~Joulin.
\newblock A new {P}oisson-type deviation inequality for {M}arkov jump processes
  with positive {W}asserstein curvature.
\newblock \emph{Bernoulli}, 15\penalty0 (2):\penalty0 532--549, 2009.

\bibitem[Joulin and Ollivier(2010)]{MR2683634}
A.~Joulin and Y.~Ollivier.
\newblock Curvature, concentration and error estimates for {M}arkov chain
  {M}onte {C}arlo.
\newblock \emph{Ann. Probab.}, 38\penalty0 (6):\penalty0 2418--2442, 2010.

\bibitem[Kallenberg(2002)]{MR1876169}
O.~Kallenberg.
\newblock \emph{Foundations of modern probability}.
\newblock Probability and its Applications (New York). Springer-Verlag, New
  York, second edition, 2002.

\bibitem[Khasminskii(2012)]{MR2894052}
R.~Khasminskii.
\newblock \emph{Stochastic stability of differential equations}, volume~66 of
  \emph{Stochastic Modelling and Applied Probability}.
\newblock Springer, Heidelberg, second edition, 2012.
\newblock With contributions by G. N. Milstein and M. B. Nevelson.

\bibitem[Kulik(2015)]{kulik2015introduction}
A.~Kulik.
\newblock \emph{Introduction to Ergodic rates for Markov chains and processes:
  with applications to limit theorems}, volume~2.
\newblock Universit{\"a}tsverlag Potsdam, 2015.

\bibitem[Lindvall and Rogers(1986)]{MR841588}
T.~Lindvall and L.~C.~G. Rogers.
\newblock Coupling of multidimensional diffusions by reflection.
\newblock \emph{Ann. Probab.}, 14\penalty0 (3):\penalty0 860--872, 1986.

\bibitem[Majka(2015)]{2015arXiv150908816M}
M.~Majka.
\newblock {Coupling and exponential ergodicity for stochastic differential
  equations driven by L\'evy processes}.
\newblock \emph{ArXiv e-prints}, Sept. 2015.

\bibitem[Malrieu(2003)]{MR1970276}
F.~Malrieu.
\newblock Convergence to equilibrium for granular media equations and their
  {E}uler schemes.
\newblock \emph{Ann. Appl. Probab.}, 13\penalty0 (2):\penalty0 540--560, 2003.

\bibitem[Malyshkin(2000)]{MR1967786}
M.~N. Malyshkin.
\newblock Subexponential estimates for the rate of convergence to the invariant
  measure for stochastic differential equations.
\newblock \emph{Teor. Veroyatnost. i Primenen.}, 45\penalty0 (3):\penalty0
  489--504, 2000.

\bibitem[Mattingly(2002)]{MR1937652}
J.~C. Mattingly.
\newblock Exponential convergence for the stochastically forced
  {N}avier-{S}tokes equations and other partially dissipative dynamics.
\newblock \emph{Comm. Math. Phys.}, 230\penalty0 (3):\penalty0 421--462, 2002.

\bibitem[M{\'e}l{\'e}ard(1996)]{MR1431299}
S.~M{\'e}l{\'e}ard.
\newblock Asymptotic behaviour of some interacting particle systems;
  {M}c{K}ean-{V}lasov and {B}oltzmann models.
\newblock In \emph{Probabilistic models for nonlinear partial differential
  equations ({M}ontecatini {T}erme, 1995)}, volume 1627 of \emph{Lecture Notes
  in Math.}, pages 42--95. Springer, Berlin, 1996.

\bibitem[Meyn and Tweedie(2009)]{MR2509253}
S.~Meyn and R.~L. Tweedie.
\newblock \emph{Markov chains and stochastic stability}.
\newblock Cambridge University Press, Cambridge, second edition, 2009.
\newblock With a prologue by Peter W. Glynn.

\bibitem[Meyn and Tweedie(1992)]{MR1174380}
S.~P. Meyn and R.~L. Tweedie.
\newblock Stability of {M}arkovian processes. {I}. {C}riteria for discrete-time
  chains.
\newblock \emph{Adv. in Appl. Probab.}, 24\penalty0 (3):\penalty0 542--574,
  1992.

\bibitem[Meyn and Tweedie(1993)]{MR1234295}
S.~P. Meyn and R.~L. Tweedie.
\newblock Stability of {M}arkovian processes. {III}. {F}oster-{L}yapunov
  criteria for continuous-time processes.
\newblock \emph{Adv. in Appl. Probab.}, 25\penalty0 (3):\penalty0 518--548,
  1993.

\bibitem[Mitrophanov(2003)]{MR2012680}
A.~Y. Mitrophanov.
\newblock Stability and exponential convergence of continuous-time {M}arkov
  chains.
\newblock \emph{J. Appl. Probab.}, 40\penalty0 (4):\penalty0 970--979, 2003.

\bibitem[Nummelin and Tuominen(1983)]{MR711187}
E.~Nummelin and P.~Tuominen.
\newblock The rate of convergence in {O}rey's theorem for {H}arris recurrent
  {M}arkov chains with applications to renewal theory.
\newblock \emph{Stochastic Process. Appl.}, 15\penalty0 (3):\penalty0 295--311,
  1983.

\bibitem[{Pillai} and {Smith}(2014)]{2014arXiv1405.0182P}
N.~S. {Pillai} and A.~{Smith}.
\newblock {Ergodicity of Approximate MCMC Chains with Applications to Large
  Data Sets}.
\newblock \emph{ArXiv e-prints}, May 2014.

\bibitem[Revuz and Yor(1999)]{MR1725357}
D.~Revuz and M.~Yor.
\newblock \emph{Continuous martingales and {B}rownian motion}, volume 293 of
  \emph{Grundlehren der Mathematischen Wissenschaften [Fundamental Principles
  of Mathematical Sciences]}.
\newblock Springer-Verlag, Berlin, third edition, 1999.

\bibitem[Roberts and Rosenthal(1996)]{MR1423462}
G.~O. Roberts and J.~S. Rosenthal.
\newblock Quantitative bounds for convergence rates of continuous time {M}arkov
  processes.
\newblock \emph{Electron. J. Probab.}, 1:\penalty0 no.\ 9, approx.\ 21 pp.\
  (electronic), 1996.

\bibitem[Roberts et~al.(1998)Roberts, Rosenthal, and Schwartz]{MR1622440}
G.~O. Roberts, J.~S. Rosenthal, and P.~O. Schwartz.
\newblock Convergence properties of perturbed {M}arkov chains.
\newblock \emph{J. Appl. Probab.}, 35\penalty0 (1):\penalty0 1--11, 1998.

\bibitem[{Rudolf} and {Schweizer}(2015)]{2015arXiv150304123R}
D.~{Rudolf} and N.~{Schweizer}.
\newblock {Perturbation theory for Markov chains via Wasserstein distance}.
\newblock \emph{ArXiv e-prints}, Mar. 2015.

\bibitem[Shardlow and Stuart(2000)]{MR1756418}
T.~Shardlow and A.~M. Stuart.
\newblock A perturbation theory for ergodic {M}arkov chains and application to
  numerical approximations.
\newblock \emph{SIAM J. Numer. Anal.}, 37\penalty0 (4):\penalty0 1120--1137,
  2000.

\bibitem[Sznitman(1991)]{MR1108185}
A.-S. Sznitman.
\newblock Topics in propagation of chaos.
\newblock In \emph{\'{E}cole d'\'{E}t\'e de {P}robabilit\'es de {S}aint-{F}lour
  {XIX}---1989}, volume 1464 of \emph{Lecture Notes in Math.}, pages 165--251.
  Springer, Berlin, 1991.

\bibitem[Tuominen and Tweedie(1994)]{MR1285459}
P.~Tuominen and R.~L. Tweedie.
\newblock Subgeometric rates of convergence of {$f$}-ergodic {M}arkov chains.
\newblock \emph{Adv. in Appl. Probab.}, 26\penalty0 (3):\penalty0 775--798,
  1994.

\bibitem[Veretennikov(1997)]{MR1472961}
A.~Y. Veretennikov.
\newblock On polynomial mixing bounds for stochastic differential equations.
\newblock \emph{Stochastic Process. Appl.}, 70\penalty0 (1):\penalty0 115--127,
  1997.

\bibitem[Villani(2009)]{MR2459454}
C.~Villani.
\newblock \emph{Optimal transport}, volume 338 of \emph{Grundlehren der
  Mathematischen Wissenschaften [Fundamental Principles of Mathematical
  Sciences]}.
\newblock Springer-Verlag, Berlin, 2009.
\newblock Old and new.

\bibitem[von Renesse and Sturm(2005)]{MR2142879}
M.-K. von Renesse and K.-T. Sturm.
\newblock Transport inequalities, gradient estimates, entropy, and {R}icci
  curvature.
\newblock \emph{Comm. Pure Appl. Math.}, 58\penalty0 (7):\penalty0 923--940,
  2005.

\bibitem[Wang(2016)]{JWang}
J.~Wang.
\newblock {$L^p$}-{W}asserstein distance for stochastic differential equations
  driven by {L}\'evy processes.
\newblock \emph{Bernoulli}, 22\penalty0 (3):\penalty0 1598--1616, 2016.

\bibitem[Zimmer(2017{\natexlab{a}})]{Zimmer}
R.~Zimmer.
\newblock Kantorovich contractions with explicit constants for diffusion
  semigroups.
\newblock \emph{PhD Thesis, University of Bonn}, 2017{\natexlab{a}}.

\bibitem[Zimmer(2017{\natexlab{b}})]{zimmer16}
R.~Zimmer.
\newblock Explicit contraction rates for a class of degenerate and
  infinite-dimensional diffusions.
\newblock \emph{Stoch. Partial Differ. Equ. Anal. Comput.}, 5\penalty0
  (3):\penalty0 368--399, 2017{\natexlab{b}}.

\end{thebibliography}
\end{document}